\documentclass[a4paper,12pt]{amsart}
\usepackage{amsfonts}
\usepackage{amsmath}
\usepackage{amssymb}
\usepackage{graphicx}
\usepackage{enumerate}
\usepackage{multicol}
\usepackage{mathrsfs}
\usepackage[usenames,dvipsnames]{pstricks}
\usepackage{epsfig}
\usepackage{pst-grad} 
\usepackage{pst-plot} 
\usepackage[hmargin=3cm,vmargin=2.5cm]{geometry}
\usepackage{rotating}
\usepackage[utf8]{inputenc}
\usepackage{mathtools}
\usepackage{amscd}

\usepackage{amsmath}
\usepackage{amscd}
\usepackage{latexsym}
\usepackage{amssymb}
\usepackage{amsthm}
\usepackage{delarray}
\newtheorem{theorem}{Theorem}[subsection]
\newtheorem{proposition}[theorem]{Proposition}
\newtheorem{corollary}[theorem]{Corollary}
\newtheorem{definition}[theorem]{Definition}
\newtheorem{lemma}[theorem]{Lemma}

\newtheorem{example}{Example}
\numberwithin{equation}{section}

\newcommand{\n}{\mathscr N}

\DeclareMathOperator{\End}{End}
\DeclareMathOperator{\Aut}{Aut}
\DeclareMathOperator{\Ad}{Ad}
\DeclareMathOperator{\Pin}{Pin}
\DeclareMathOperator{\Spin}{Spin}

\DeclareMathOperator{\la}{\langle}
\DeclareMathOperator{\ra}{\rangle}

\DeclareMathOperator{\Id}{Id}
\DeclareMathOperator{\Cl}{\rm{Cl}}
\DeclareMathOperator{\GL}{\rm{GL}}
\DeclareMathOperator{\Orth}{\rm{O}}
\DeclareMathOperator{\SO}{\rm{SO}}

\begin{document}
\title[Complete classification]{Complete classification of pseudo $H$-type algebras: II}

\author[Kenro Furutani, Irina Markina]  
{Kenro Furutani, Irina Markina}

\date{\today}
\thanks{
The first author was partially
supported by the Grant-in-aid for Scientific Research (C) No. 26400124, JSPS.,
and the second author was partially supported by ISP project 239033/F20 of Norwegian Research Council, as well as EU FP7 IRSES program STREVCOMS, grant  no. PIRSES-GA-2013-612669}

\subjclass[2010]{Primary 17B60, 17B30, 17B70, 22E15}

\keywords{Clifford module, nilpotent 2-step Lie algebra, pseudo $H$-type algebras, Lie algebra isomorphism, scalar product}

\address{K.~Furutani:  Department of Mathematics, Faculty of Science 
and Technology, Science University of Tokyo, 2641 Yamazaki, Noda, Chiba (278-8510), Japan}
\email{furutani\_kenro@ma.noda.tus.ac.jp}

\address{I.~Markina: Department of Mathematics, University of Bergen, P.O.~Box 7803,
Bergen N-5020, Norway}
\email{irina.markina@uib.no}

\begin{abstract}
We classify a class of 2-step nilpotent Lie algebras related to the representations of the Clifford algebras in the following way. Let $J\colon \Cl(\mathbb R^{r,s})\to\End(U)$ be a representation of the Clifford algebra $\Cl(\mathbb R^{r,s})$ generated by the pseudo Euclidean vector space $\mathbb R^{r,s}$. Assume that the Clifford module $U$ is endowed with a bilinear symmetric non-degenerate real form $\la\cdot\,,\cdot\ra_U$ making the linear map $J_z$ skew symmetric for any $z\in\mathbb R^{r,s}$. The Lie algebras and the Clifford algebras are related by $\la J_zv,w\ra_U=\la z,[v,w]\ra_{\mathbb R^{r,s}}$, $z\in \mathbb R^{r,s}$, $v,w\in U$. We detect the isomorphic and non-isomorphic Lie algebras according to the dimension of $U$ and the range of the non-negative integers $r,s$. 
\end{abstract}

\maketitle
\tableofcontents


\section{Introduction}


The present work is a continuation of~\cite{FM1} and studies 2-step nilpotent graded Lie algebras associated to the representations of Clifford algebras. Let $\Cl_{r,s}$ be the Clifford algebra generated by the pseudo Euclidean vector space $\mathbb R^{r,s}$ and let $J\colon \Cl_{r,s}\to\End(U)$ be its representation. Assume that there exists a non-degenerate symmetric bi-linear form $\la\cdot\,,\cdot\ra_U$ on the Clifford module $U$ such that
$\la J_zx,y\ra_U+\la x,J_zy\ra_U=0$ for all $z\in\mathbb R^{r,s}$ and $x,y\in U$. The pair $(U,\la\cdot\,,\cdot\ra_U)$ is called an {\it admissible module}. The set $U\oplus\mathbb R^{r,s}$, endowed with the Lie bracket defined by $\la z,[x,y]\ra_{\mathbb R^{r,s}}=\la J_zx,y\ra_U$  for $x,y\in U$ and zero otherwise, is called a {\it pseudo $H$-type Lie algebra} and is denoted by $\n_{r,s}(U)$. The pseudo $H$-type Lie algebras $\n_{r,0}(U)$ were introduced in~\cite{Ka80} and their generalisations $\n_{r,s}(U)$ appeared in~\cite{Ci,GKM}. The pseudo $H$-type Lie algebras were actively studied for instance in~\cite{Barbano,CD,FM,FM1,Kaplan81}. They provide a setting for the study of sub-elliptic, hypo-elliptic and Grushin type operators, see for instance~\cite{BieskeG,CChFI} and it is an important particular case of the extended Poincar\'e Lie super algebras~\cite{AC,AS1}. The Lie groups of pseudo $H$-type Lie algebras constitute a source of interesting examples of sub-Riemannian manifolds~\cite{Bell,Hein}, nil-manifolds~\cite{Barco,KOR}, iso-spectral but non-diffeomorphic manifolds~\cite{Bauer}, Damek-Ricci harmonic spaces~\cite{BTV}, symmetric spaces of rank one,~\cite{CDKR,Pansu}.

The authors considered in~\cite{FM1} the classification of pseudo $H$-type Lie algebras whose constructions are based on the minimal dimensional admissible modules $V^{r,s}_{min}$. It has been showen that two Lie algebras $\n_{r,s}(V^{r,s}_{min})$ and $\n_{\tilde r,\tilde s}(V^{\tilde r,\tilde s}_{min})$ are never isomorphic unless $r= \tilde r$ and $s=\tilde s$, or $r=\tilde s$ and $s=\tilde r$. Among the couples $\n_{r,s}(V^{r,s}_{min})$ and $\n_{s,r}(V^{s,r}_{min})$ there are isomorphic and non-isomorphic Lie algebras. 

The present paper is a continuation of~\cite{FM1} and it finishes the classification of the pseudo $H$-type Lie algebras $\n_{r,s}(U)$, where $U$ is not necessary minimal dimensional admissible module.  The first step of the classification depends on the fact whether the Clifford algebra $\Cl_{r,s}$ is simple or not. If the Clifford algebra $\Cl_{r,s}$ is simple, then the Lie algebra $\n_{r,s}(U)$ for $r\neq 3 \,(\text{mod}~4)$ is uniquely defined by the dimension of the admissible module $U$ and does not depend on the choice of the scalar product on $U$. As a consequence in this case we obtain that the Lie algebras $\n_{r,s}(U)$ and $\n_{s,r}(\tilde U)$ are isomorphic if $\dim(U)=\dim(\tilde U)$. If $r=3 \,(\text{mod}~4)$ then the Lie algebra $\n_{r,s}(U)$ depends on the choice of the scalar product on each minimal dimensional component $(V^{r,s}_{min})_i$ in the decomposition $U=\oplus (V^{r,s}_{min})_i$. If $r-s=3 \,(\text{mod}~4)$, then Clifford algebras $\Cl_{r,s}$ are not simple and the classification is more complicate. Recall, that the Lie algebras $\n_{r,0}(U)$ for $r=3 \,(\text{mod}~4)$ are defined only by number of non-equivalent irreducible terms in the decomposition of $U$ into the direct sum of irreducible submodules and any irreducible Clifford module is actually an admissible module, see~\cite{BTV,CDKR}. For the pseudo $H$-type Lie algebras $\n_{r,s}(U)$, $s>0$, and $r=3 \,(\text{mod}~4)$ the classification is more subtle and depends not only on the number of different minimal dimensional modules, but also on the choice of the scalar product on them. These phenomena come from the signature of the scalar product
restricted to the ``common 1-eigenspace'' of a set of maximal number of mutually 
commuting symmetric isometric involutions of the Clifford action on the minimal dimensional module. It is also closely related to the existence or non-existence of a special type of an automorphism of the Lie algebra $\n_{r,s}(U)$ which is identity on the centre.

The structure of the paper is the following. We recall basic properties of Clifford algebras, such as, periodicity, the system of involutions, the structure of admissible modules, and other information needed to complete classification of pseudo H-type Lie algebras in Section~\ref{sec:ClifAlg}. Section~\ref{sec:LieAlgebras} is dedicated to the description of pseudo H-type Lie algebras and the structure of their isomorphisms and automorphisms. 
The main result is contained in Theorems~\ref{main theorem 1}--\ref{main theorem 3}, see Section~\ref{sec:Classification}. In Section~\ref{app}, we present Tables~\ref{tab:com0r47}--\ref{tab:block3} needed to determine important properties of minimal admissible modules for basic cases~\eqref{basic cases}, which are summarised in Table~\ref{t:dim} in Section~\ref{sec:dimensions}. 


\section{Clifford algebras and admissible modules}\label{sec:ClifAlg}


In the section we collect the information 
about Clifford algebras and their admissible modules 
that we need for the classification of pseudo $H$-type Lie algebras.


\subsection{Definitions of Clifford algebras}\label{sec:DefClifAlg}


We denote by $\mathbb R^{r,s}$ the space $\mathbb R^{k}$, $r+s=k$, with
the non-degenerate quadratic form 
$
Q_{r,s}(z)=\sum_{i=1}^{r}z_{i}^{2}-\sum_{j=1}^{s}z_{r+j}^{2}$, $z\in \mathbb R^n
$ of the signature $(r,s)$.
The non-degenerate bi-linear form obtained from $Q_{r,s}$ by polarization
is denoted by $\la\cdot\,,\cdot\ra_{r,s}$. 
We call the form
$\la\cdot\,,\cdot\ra_{r,s}$ a {\it scalar product}. 
A vector $z\in\mathbb R^{r,s}$ is called
{\it positive} if $\la z,z\ra_{r,s}>0$,
{\it negative} if $\la z,z\ra_{r,s}<0$, and {\it null} 
if $\la z,z\ra_{r,s}=0$. 
We use the orthonormal basis
$\{z_1,\ldots,z_r,z_{r+1},\ldots,r_{r+s}\}$ 
for $\mathbb R^{r,s}$, where the basis vectors $z_1,\ldots,z_r$ are
positive and $z_{r+1},\ldots,z_{r+s}$ are negative. 

Let $\Cl_{r,s}$ be the real Clifford algebra 
generated by $\mathbb{R}^{r,s}$, that is the quotient of the tensor algebra 
$$
\mathcal{T}(\mathbb{R}^{r+s})=\mathbb{R}\oplus\left(\mathbb{R}^{r+s}\right)\oplus 
\left(\stackrel{2}\otimes\mathbb{R}^{r+s}\right)\oplus \left(\stackrel{3}\otimes\mathbb{R}^{r+s}\right)\oplus\cdots, 
$$
divided by the two-sided
ideal $I_{r,s}$ which is generated by the elements of the form
$z\otimes z+\la z,z\ra_{r,s}$, $z\in\mathbb{R}^{r+s}$.
The explicit determination of the Clifford algebras 
is given in~\cite{ABS} and they are isomorphic 
to matrix algebras 
$\mathbb{R}(n)$, 
$\mathbb{R}(n)\oplus\mathbb{R}(n)$, 
$\mathbb{C}(n)$, 
$\mathbb{H}(n)$ or 
$\mathbb{H}(n)\oplus\mathbb{H}(n)$ where 
the size $n$ is determined by $r$ and $s$, see~\cite{LawMich}. 

Given an algebra homomorphism $\widehat J\colon
\Cl_{r,s}\to\End(U)$, we call the space $U$ a Clifford module and
the operator $J_{\phi}$ a {\it Clifford action} or a {\it representaion map}
of an element $\phi\in \Cl_{r,s}$. 
If there is a map
$$
\begin{array}{cccccc}
J\colon &\mathbb R^{r,s}&\to&\End(U)
\\
&z &\mapsto &J_z,
\end{array}
$$ 
satisfying $J^2_z=-\la z,z\ra_{r,s}\Id_{U}$ for an
arbitrary $z\in\mathbb R^{r,s}$, then $J$ can be uniquely extended to an algebra homomorphism 
$\widehat J$ by 
the universal property, see, for instance~\cite{Hu,Lam,LawMich}. 
Even though the representation matrices of the Clifford algebras
$\Cl_{r,s}$, and the Clifford modules $U$ are given over the fields $\mathbb{R}$, $\mathbb{C}$ or
$\mathbb{H}$, we refer to $\Cl_{r,s}$ as a real algebra and $U$ as a real vector space.
The dimension of $U$ is a multiple of $n$ over the
corresponding fields $\mathbb{R}$, $\mathbb{C}$ or
$\mathbb{H}$.

If $r-s\not\equiv 3 \,(\text{mod}~4)$, then $\Cl_{r,s}$ is a simple algebra. 
In this case there is only one
irreducible module $U=V_{irr}^{r,s}$ of dimension $n$. 
If $r-s\equiv 3 \,(\text{mod}~4)$, then the algebra $\Cl_{r,s}$ is not simple, 
and there are two non-equivalent irreducible modules.
They can be distinguished by the action of the ordered volume form 
$\Omega^{r,s}=\prod_{k=1}^{r+s}z_k$. In fact, the elements $\frac{1}{2}\big(\Id\mp\Omega^{r,s}\big)$
act as an identity operator on the Clifford module, so 
$J_{\Omega^{r,s}}\equiv\pm \Id$. 
Thus we denote by $V^{r,s}_{irr;\pm}$ two non-equivalent irreducible Clifford modules 
on which the action of the volume form is given by $J_{\Omega^{r,s}}=\prod_{k=1}^{r+s}J_{z_k}\equiv\pm \Id$.
\begin{proposition}{\em\cite{LawMich}}\label{prop:reducibility}
Clifford modules are completely reducible. Namely, let $U$ 
be a Clifford module, then it can be
decomposed into irreducible modules:
\begin{equation}\label{eq:isotypic}
U=
\begin{cases} 
\stackrel{p}\oplus V_{irr}^{r,s},\quad&\text{if}\quad r-s\not\equiv 3\,($\text{\em mod}$~4),
\\
\big(\stackrel{p_+}\oplus V_{irr;+}^{r,s}\big)\oplus
\big(\stackrel{p_-}\oplus V_{irr;-}^{r,s}\big),
\quad&\text{if}\quad r-s\equiv 3 \,($\text{\em mod}$~4).
\end{cases}
\end{equation}
The numbers $p$, $p_+,p_-$ are uniquely determined by the dimension of $U$.
\end{proposition}
The module $U= \stackrel{p}\oplus V_{irr}^{r,s}$ is called {\it isotypic} and the second one in~\eqref{eq:isotypic} is non-isotypic.


\subsection{Admissible modules}


\begin{definition}\cite{Ci} A module $U$ of the Clifford algebra $\Cl_{r,s}$ 
is called admissible if there is a scalar product $\la\cdot\,,\cdot\ra_{U}$ on
$U$ such that
\begin{equation}\label{admissibility condition 1}
\la J_{z}x,y\ra_{U}+\la x,J_{z}y\ra_U=0,\quad\text{for all}\  \ x,y\in U\ \text{and}\  z\in \mathbb{R}^{r,s}.
\end{equation}
\end{definition}
We write $(U,\la \cdot\,,\cdot\ra_{U})$ for
an admissible module to emphasise the scalar product $\la\cdot\,,\cdot\ra_{U}$
and call it an {\it admissible scalar product}. 
We collect properties of admissible modules in several propositions.
\begin{proposition}\label{properties of admissible module} Let $\Cl_{r,s}$ be the Clifford algebra generated by the space $\mathbb R^{r,s}$.
\begin{itemize}
\item[(1)] {If $\la
\cdot\,,\cdot\ra_{U}$ 
is an admissible scalar product for $\Cl_{r,s}$, then $K\la \cdot\,,\cdot\ra_{U}$ 
is also admissible for any constant $K\neq 0$. We can assume that $K=\pm 1$ by normalisation of the scalar products.
}
\item[(2)] {Let $(U,\la\cdot\,,\cdot\ra_U)$ be an admissible module for $\Cl_{r,s}$ and let $(U_1,\la\cdot\,,\cdot\ra_{U_1})$ be such that $U_1$ is a submodule of $U$ and $\la\cdot\,,\cdot\ra_{U_1}$ is the
restriction of $\la\cdot\,,\cdot\ra_U$ to $U_{1}$. Then the orthogonal complement
${U_{1}}^{\perp}=\{x\in U\mid\ \la x,y\ra_{U}=0,\ \text{for all}\  y\in U_{1}\}$
with the scalar product obtained by the restriction of $\la\cdot\,,\cdot\ra_U$ to ${U_{1}}^{\perp}$ 
is also an admissible module.
}
\item[(3)] {Condition~\eqref{admissibility condition 1} and the
property $J^{\,2}_{z}=-\la z,z\ra_{r,s}\Id_U$ imply
\begin{equation}\label{admissibility condition 2}
\la J_{z}x,J_{z}y\ra_{U}=\la z,z\ra_{r,s}\la x,y\ra_{U}.
\end{equation}
}
\item[(4)] {Relation~\eqref{admissibility condition 2} leads to the following:
if $z\in\mathbb{R}^{r,s}$ is positive, then
$$
\hskip-0.6cm\la v,v\ra_U>0\ \text{implies}\ \la J_{z}v,J_{z}v\ra_U>0,\ \text{and}\ 
\la v,v\ra_U<0\ \text{implies}\ \la J_{z}v,J_{z}v\ra_U<0.
$$
In other words the map $J_z\colon U\to U$ is an isometry for $\la z,z\ra_{r,s}=1$. If $z\in\mathbb{R}^{r,s}$ is negative, then
$$
\hskip-0.6cm\la v,v\ra_U>0\ \text{implies}\ \la J_{z}v,J_{z}v\ra_U<0,\ \text{and}\ 
\la v,v\ra_U<0\ \text{implies}\ \la J_{z}v,J_{z}v\ra_U>0.
$$
and the map $J_z\colon U\to U$ is an anti-isometry for $\la z,z\ra_{r,s}=-1$.
}
\item[(5)] {If $s>0$, then any admissible module $(U,\la \cdot\,,\cdot\ra_{U})$ of $\Cl_{r,s}$
is {\it neutral}, i.e., $\dim U=2l$, $l\in\mathbb N$, and $U$ it is isometric to
$\mathbb{R}^{l,l}$, see~{\em \cite[Proposition 2.2]{Ci}}.
}
\item[(6)] {If $s=0$, then any Clifford module of $\Cl_{r,0}$
    can be made into admissible with positive definite or negative
    definite scalar product, see~{\em \cite{Hu}}.
}
\end{itemize}
\end{proposition}

Proposition~\ref{r-s = 0,1,2 mod 4} describes the relation between irreducible and admissible modules. This relation depends on the signature $(r,s)$ of the generating space $\mathbb R^{r,s}$ for the Clifford algebra $\Cl_{r,s}$.
An admissible module of the minimal possible dimension is called a {\it minimal admissible module}. 

\begin{proposition}\label{r-s = 0,1,2 mod 4}\cite[Theorem 3.1]{Ci}\cite[Proposition 1]{FM1}
Let $\Cl_{r,s}$ be the Clifford algebra generated by the space $\mathbb R^{r,s}$.
\begin{itemize}
\item[ (1)] {If $s=0$, then any irreducible Clifford module is minimal
admissible with respect to a positive definite or a 
negative definite scalar product. 
}
\item[(2)] {If $r-s\equiv 0,1,2 \,(\text{\em mod}~4)$, 
then a unique irreducible module $V^{r,s}_{irr}$ is not necessary admissible. The following situations are possible: 
\begin{itemize}
\item[(2-1)] {The irreducible module $V^{r,s}_{irr}$ is minimal admissible or,}
\item[(2-2)] {The irreducible module $V^{r,s}_{irr}$ is not admissible, but the
direct sum $V^{r,s}_{irr}\oplus V^{r,s}_{irr}$ is minimal admissible.}
\end{itemize}
}
\item[(3)] {
If $r-s\equiv 3 \,(\text{\em mod})~4$, then for two non-equivalent irreducible modules $V_{irr;\pm}^{r,s}$ the following can occur:
\begin{itemize}
\item[(3-1)] {Each irreducible module $V_{irr;\pm}^{r,s}$ is minimal admissible. The index $s$ must be even in this case.}
\item[(3-2)] {None of the irreducible modules $V_{irr;\pm}^{r,s}$ is
admissible. It can happened for even and odd values of $s$. 
\begin{itemize}
\item[(3-2-1)] {If $s$ is even, then $V_{irr;+}^{r,s}\oplus V_{irr;+}^{r,s}$, $V_{irr;-}^{r,s}\oplus V_{irr;-}^{r,s}$ are minimal admissible modules, and the module $V_{irr;+}^{r,s}\oplus V_{irr;-}^{r,s}$ is not admissible.}
\item[(3-2-2)]{ If $s$ is odd, then the module $V_{irr;+}^{r,s}\oplus V_{irr;-}^{r,s}$ is minimal admissible and neither $V_{irr;+}^{r,s}\oplus V_{irr;+}^{r,s}$ nor $V_{irr;-}^{r,s}\oplus V_{irr;-}^{r,s}$ is admissible.}
\end{itemize}
}
\end{itemize}
}
\end{itemize}
\end{proposition}

We emphasize the following corollary, see also Table~\ref{t:dim} and the remark after it.
 \begin{corollary}\label{r-s = 3 mod 4 (2)}
There are two minimal admissible modules only if $r-s\equiv
3\,(\text{\em mod}~4)$
and $s$ is even. 
We distinguish two cases:
\begin{itemize}
\item[(3-1)] Each irreducible module is minimal admissible: 
$V_{min;+}^{r,s}=V_{irr;+}^{r,s}$ and
$V_{min;-}^{r,s}=V_{irr;-}^{r,s}$. In this case 
$r\equiv 3\,(\text{\em mod}~4)$, $s\equiv 0\,(\text{\em
  mod}~4)$, or $r\equiv 1\,(\text{\em mod}~8)$, $s\equiv 2\,(\text{\em mod}~8)$,
or $r\equiv 5\,(\text{\em mod}~8)$, $s\equiv 6\,(\text{\em mod}~8)$.
\item[(3-2-1)] Direct sums of irreducible modules are minimal 
admissible: $V_{min;+}^{r,s}=V_{irr;+}^{r,s}\oplus V_{irr;+}^{r,s}$ 
and $V_{min;-}^{r,s}=V_{irr;-}^{r,s}\oplus V_{irr;-}^{r,s}$. It happens if
$r\equiv 1\,(\text{\em mod}~8)$, $s\equiv 6\,(\text{\em mod}~8)$,
or $r\equiv 5\,(\text{\em mod}~8)$, $s\equiv 2\,(\text{\em mod}~8)$.
\end{itemize}
\end{corollary}


\subsection{Mutually commuting isometric involutions}


Recall that a linear map $\Lambda$ defined on a vector space $U$ with
a scalar
product $\la\cdot\,,\cdot\ra_{U}$ is called 
{\it symmetric} with respect to the scalar product $\la\cdot\,,\,\cdot\ra_{U}$, 
if $\la \Lambda x,y\ra_{U}=\la x,\Lambda y\ra_{U}$, 
{\it isometric $($or positive$)$} if it maps positive vectors to
positive vectors and negative vectors to negative vectors
and {\it anti-isometric $($or negative$)$} if it reverses the positivity and negativity of the vectors.
Let $J_{z_i}$ be representation maps for an orthonormal basis
$\{z_1,\ldots,z_{r+s}\}$ of~$\mathbb R^{r,s}$. 
The simplest isometric involutions, written as a product of the maps
$J_{z_i}$, are one of the following forms:
\begin{equation}\label{eq:basicInvol}
\begin{cases}
\mathcal{P}_1=J_{z_{i_1}} J_{z_{i_2}} J_{z_{i_3}} J_{z_{i_4}},\ 
  \text{where all}\ z_{i_k}~\text{are either positive or negative},
  \\
\mathcal{P}_{2}= J_{z_{i_1}} J_{z_{i_2}} J_{z_{i_3}} J_{z_{i_4}},\ \text{where two $z_{i_l}$ are positive and two are negative},
\\
\mathcal{P}_{3}=J_{z_{i_1}} J_{z_{i_2}} J_{z_{i_3}},\qquad \text{where all three
 $z_{i_{k}}$ are positive},
 \\
\mathcal{P}_{4}=J_{z_{i_1}} J_{z_{i_2}} J_{z_{i_3}},\qquad\text{where one $z_{i_l}$ is positive
  and two are negative}.
  \end{cases}
\end{equation}
The product involutions of types $\mathcal{P}_{3}$ and $\mathcal{P}_{4}$ is not an involution, meanwhile the product of involutions of other types is again an involution.
For a given minimal admissible module $V_{min}^{r,s}$, 
we denote by $PI_{r,s}$ a set of the maximal
number of mutually commuting symmetric isometric 
involutions of the forms~\eqref{eq:basicInvol} 
and such that none of them is a product
of other involutions in $PI_{r,s}$. 
The set $PI_{r,s}$ is not unique, while 
the number of involutions $p_{r,s}=\#\{PI_{r,s}\}$ in $PI_{r,s}$ is unique for the given signature~$(r,s)$.
The set $PI_{r,s}$ can be chosen equal for the modules with the  
admissible scalar product of opposite signs, or for the minimal admissible modules, based on the two non-equivalent Clifford modules.
The ordering on the set $PI_{r,s}$ can be made, if necessary, in such a way 
that at most one involution of the type $\mathcal{P}_{3}$ or $\mathcal{P}_{4}$ is included in $PI_{r,s}$ and it is the last one, see Sections~\ref{sec:Bott} and~\ref{app}. 

We also need 
a set $CO_{r,s}$ of {\it complementary} operators to
$PI_{r,s}$, which are 
products of the maps $J_{z_{i}}$ and they are
ordered according to the
ordering in $PI_{r,s}$ such that 
$$
C_{k}P_i=P_iC_k\quad\text{for}\quad i<k,\quad\text{and}\quad C_kP_{k}=-P_kC_k,\quad P_k\in PI_{r,s},\ C_k\in CO_{r,s}.
$$ 
The complementary operators can be isometric or anti-isometric. They guarantee that all the involutions from $PI_{r,s}$ have their both eigenspaces as subspaces of $V^{r,s}_{min}$. Note that if $r-s\equiv 3\,(\text{mod}~4)$ and $s$ is even, then the last involution of type $\mathcal{P}_{3}$ or $\mathcal{P}_{4}$ allows to distinguish  the modules $V^{r,s}_{min;+}$ and $V^{r,s}_{min;-}$, see for instance the proof of Theorem~\ref{th:common_sign}. Therefore, the number of operators in $CO_{r,s}$ is different and is equal to 
\begin{align}
&\text{$p_{r,s}$, when $r-s\not\equiv 3\,(\text{mod}~4)$,
~{or}~$r-s\equiv 3\,(\text{mod}~4)$ and $s$ is odd},\label{case i}
\\
&\text{$p_{r,s}-1$, when $r-s\equiv 3\,(\text{mod}~4)$ and $s$ is even}.\label{case ii}
\end{align}
We define the subspace $E_{r,s}$ of a minimal admissible module $V_{min}^{r,s}$ 
\begin{equation*}\label{common 1-eigenspaceNco1}
\begin{array}{llc}
&E_{r,s}=\{v\in V^{r,s}_{min}\mid\ P_iv=v\ \ i\leq p_{r,s}\}\quad & \text{in the case~\eqref{case i}}
\\
&E_{r,s}=\{v\in V^{r,s}_{min}\mid\
  P_iv=v\ \ i\leq p_{r,s}-1\}
& \text{in the case~\eqref{case ii}}
\end{array}
\end{equation*}
and call it the ``{\it common 1-eigenspace}'' for the system of involutions$PI_{r,s}$. The complementary opertors show whether the common 1-eigenspace
$E_{r,s}$ is a neutral or sign definite vector space with respect to
the restriction of the admissible scalar product. 
The sets $PI_{r,s}$ and $CO_{r,s}$ are collected in Section~\ref{app} and they will be mentioned precisely when it needs to be done. In the following proposition we explain the possible interaction of involutions with the complementary operators.

\begin{proposition}{\em \cite{FM}}\label{prop:neutral}
If $P$ is an isometric symmetric involution acting on the
space $(U,\la\cdot\,,\cdot\ra_{U})$ with a neutral scalar product and $E^{\pm 1}$ are eigenspaces of $P$, then 
\begin{itemize}
\item[1.]{If $\mathcal I\colon U\to U$ is an isometry such that $P\mathcal I=-\mathcal IP$, then $E^{\pm 1}$ are neutral,}
\item[2.]{If $\mathcal A\colon U\to U$ is an anti-isometry such that $P\mathcal A=\mathcal AP$, then $E^{\pm 1}$ are neutral,}
\item[3.]{If $\mathcal A\colon U\to U$ is an anti-isometry such that $P\mathcal A=-\mathcal AP$, then $E^{\pm 1}$ are either neutral or sign definite,}
\end{itemize}
with respect to the restriction of the scalar product $\la\cdot\,,\cdot\ra_{U}$ to $E^{\pm 1}$.
\end{proposition}

Since the involutions in $PI_{r,s}$ are symmetric, 
the eigenspaces are orthogonal subspaces. 
The involutions commute, therefore, they decompose the eigenspaces of
other involutions into smaller (eigen)-subspaces. 
We give an example, that is crucial for the paper.

\begin{example}\label{ex:inv44}
The set $PI_{\mu,\nu}$ for $(\mu,\nu)=(8,0),(0,8)$ or $(4,4)$ is given by
\begin{equation*}\label{mcpi2}
T_{1}=J_{\zeta_1}J_{\zeta_2}J_{\zeta_3}J_{\zeta_4},\, T_{2}=J_{\zeta_1}J_{\zeta_2}J_{\zeta_5}J_{\zeta_6},
\,T_{3}=J_{\zeta_1}J_{\zeta_2}J_{\zeta_7}J_{\zeta_8},\,T_{4}=J_{\zeta_1}J_{\zeta_3}J_{\zeta_5}J_{\zeta_7}.
\end{equation*}
The set of complementary operators
are 
$$
\begin{array}{llllllll}
&CO_{8,0}=\{C_1=J_{\zeta_1},\ &C_2=J_{\zeta_2}J_{\zeta_3},\ &C_3=J_{\zeta_3}J_{\zeta_4},\ 
&C_4=J_{\zeta_8}\},
\\
&CO_{4,4}=\{C_1=J_{\zeta_1},\ &C_2=J_{\zeta_2}J_{\zeta_3},\ &C_3=J_{\zeta_3}J_{\zeta_4},&C_4=J_{\zeta_8}\},
\\
&CO_{0,8}=\{C_1=J_{\zeta_4}J_{\zeta_5},\ &C_2=J_{\zeta_2}J_{\zeta_3},\ &C_3=J_{\zeta_3}J_{\zeta_4},&C_4=J_{\zeta_8}\}.
\end{array}
$$ 
The module $V_{min}^{\mu,\nu}$
is decomposed into $16$ one dimensional common eigenspaces of 
four involutions $T_{i}$. Let
$v\in E_{\mu,\nu}$
and $|\la v, v\ra_{V_{min}^{\mu,\nu}}|=1$. Then other common
eigenspaces
are spanned by $J_{z_{i}}v$, $i=1,\ldots, 8$, 
and $J_{z_1}J_{z_j}v$, $j=2,\ldots,8$. Hence we have
\begin{equation}\label{eq:decom08}
V_{min}^{\mu,\nu}=E_{\mu,\nu}\bigoplus_{i=1}^{8}\,J_{\zeta_i}(E_{\mu,\nu})
\bigoplus_{j=2}^{8} J_{\zeta_1}J_{\zeta_{j}}(E_{\mu.\nu}).
\end{equation}
The value $\la v, v\ra_{V_{min}^{\mu,\nu}}$ can be $\pm 1$
according to the admissible scalar product, however 
we may assume $\la v, v\ra_{V_{min}^{\mu,\nu}}=1$ by {\em Lemma~\ref{isomorphism between opposit sign scalar product}}. 
\end{example}


\subsection{Periodicity of Clifford algebras and admissible modules}\label{sec:Bott}


We call the following lower values of signature $(r,s)$:
\begin{equation}\label{basic cases}
\begin{array}{l}
(r,s)~\text{for}~0\leq r\leq 7~\text{and $0\leq s\leq 3$},\\
(r,s)~\text{for}~0\leq r\leq 3~\text{and $4\leq s\leq 7$}, ~\text{and}\\
(r,s)~\in \{(8,0), (0,8), (4,4)\}
\end{array}
\end{equation}
the {\it basic cases}. 
Recall the periodicity of Clifford algebras: 
$$
\Cl_{r,s}\otimes \Cl_{0,8}\cong \Cl_{r,s+8},\quad
\Cl_{r,s}\otimes \Cl_{8,0}\cong \Cl_{r+8,s},\quad
\Cl_{r,s}\otimes \Cl_{4,4}\cong \Cl_{r+4,s+4},
$$
where the last one follows from $\Cl_{r,s}\otimes \Cl_{1,1}\cong
\Cl_{r+1,s+1}$, see~\cite{ABS}.
The Clifford algebras
$\Cl_{\mu,\nu}$, $(\mu,\nu)\in\{(8,0),(0,8),(4,4)\}$,
are isomorphic to $\mathbb{R}(16)$. 
The unique irreducible module is minimal admissible $V_{irr}^{\mu,\nu}=V_{min}^{\mu,\nu}$ in all the cases. They are isometric to the following spaces:
$V_{min}^{8,0} \cong \mathbb{R}^{16,0}\cong
\mathbb{R}^{0,16}$, where we fix the first isomorphism for the constructions of Lie algebras due to Lemma~\ref{isomorphism between opposit sign scalar product}, and we also have
$V_{min}^{0,8}\cong \mathbb{R}^{8,8}$ 
and $V^{4,4}_{min}\cong\mathbb{R}^{8,8}$.

\begin{proposition}\label{periodicity by tensor product}{\em \cite{FM}}
If $V^{r,s}_{min}=(V^{r,s}_{min},\la\cdot\,,\cdot\ra_{V^{r,s}_{min}})$ 
is a minimal admissible module, then $
V^{r+\mu,s+\nu}=V^{r,s}_{min}\otimes V_{min}^{\mu,\nu}$
is the minimal admissible module of $\Cl_{r+\mu,s+\nu}$, where the scalar product on $V^{r+\mu,s+\nu}$ is given by 
$\la\cdot\,,\cdot\ra_{V^{r,s}_{min}}\la\cdot\,,\cdot\ra_{V^{\mu,\nu}_{min}}$ for $(\mu,\nu)\in\{(8,0), (0,8), (4,4)\}$. 
\end{proposition}
Let $\{\zeta_1,\ldots, \zeta_8\}$ be an orthonormal basis for  
$\mathbb R^{\mu,\nu}$ and $\{z_{1},\ldots,z_{r+s}\}$ be an orthonormal basis for $\mathbb R^{r,s}$. Let $J_{\zeta_{i}}$, $i=1,\ldots,8$ and  $J_{z_j}$, $j=1,\ldots,r+s$ be the respective representation maps. 
We denote by $\Omega^{\mu,\nu}=\prod_{i=1}^{8}\zeta_{i}$ 
the ordered volume form for $\Cl_{\mu,\nu}$ for $(\mu,\nu)\in\{(8,0),(0,8),(4,4)\}$. Set
\begin{eqnarray*}
\hat J_{z_j} & = & J_{z_j}\otimes J_{\Omega^{\mu,\nu}} \quad\text{for}\quad j=1,\ldots,r+s,
\\
\hat J_{\zeta_\alpha} & = & \Id_{V^{r,s}}\otimes J_{\zeta_{i}} \quad\text{for}\quad i=1,\ldots,8.
\end{eqnarray*}
Then the maps $\hat J_{z_j}$ and $\hat J_{\zeta_\alpha}$ 
are representations of an orthonormal basis $\{z_{j}, \zeta_{i}\}$ of $\mathbb R^{r+\mu,s+\nu}$
on the space $V^{r,s}_{min}\otimes V^{\mu,\nu}_{min}$,
as it was shown in~\cite{FM}.


\subsection{Dimensions of minimal admissible modules}\label{sec:dimensions}


The dimensions of minimal admissible modules are determined for
basic cases~\eqref{basic cases}. Then $\dim(V^{r+\mu,s+\nu}_{min})=\dim(V^{r,s}_{min})\cdot\dim(V^{\mu,\nu}_{min})=16\dim(V^{r,s}_{min})$ 
for any minimal admissible module $V^{r,s}_{min}$,
$(\mu,\nu)\in\{(8,0),(0,8),(4,4)\}$. Moreover $\dim(V^{r,s}_{min})=2^{r+s-p_{r,s}}$ where $p_{r,s}=\#\{PI_{r,s}\}$. This follows from the fact that 
minimal admissible modules are cyclic modules
and $p_{r,s}$ relations among the $2^{r+s}$ vectors 
$\{J_{z_{i_1}}J_{z_{i_{2}}}\cdots J_{z_{i_{k}}}v\}$, $v\in E_{r,s}$,
allow us to span the space $V^{r,s}_{min}$ by $2^{r+s-p_{r,s}}$ number
of linearly independent vectors.
We describe the number and the dimension
of minimal admissible modules $V^{r,s}_{min}$ in
Table~\ref{t:dim}. We indicate whether 
the scalar product restricted to the common 1-eigenspaces $E_{r,s}$ of the
involutions from $PI_{r,s}$ is neutral or sign definite, 
see Section~\ref{sec:1eigenspace} for the proof.
\begin{table}[htbp]
	\center
	\caption{Dimensions of minimal admissible modules}
\vskip0.3cm	
	\begin{tabular}{|c||c|c|c||c||c|c|c||c||c|}
\hline
		${\text{\small 8}} $&$ {\text{\small{16}}^{\pm}}$&${\text{\small 32}}^{\pm}$&${\text{\small{64}}}^{\pm}
		$&${\text{\small{64}$_{\times 2}^{\pm}$}}$&${\text{\small{128}}^{\pm}}$&${\text{\small{128}}^{\pm}}$&${\text{{\small{128}}}^{\pm}}$&$
		{\text{{\small{128}$_{\times 2}^{\pm}$}}} $&${\text{\small{256}}^{\pm}}$
\\
\hline\hline
		${\text{\small 7}}$ &$ {\text{\small{16}}^{N}}$&${\text{\small{32}}^{N}}$&$
		{\mathbf{\small{64}}^{N}} $&${\text{\small{64}}}^{\pm}
		$&${\mathbf{\small{128}}^{N}}$&${\mathbf{\small{128}}^{N}}  $&${\mathbf{\small{128}}^{N}}$&$ {\text{\small{128}}^{\pm}} $&${\text{\small{256}}^{N}}$
\\
\hline
		${\mathbf{\small 6}}$ &${\text{\small{16}}^{N}}$&${\text{\small{16}$_{\times 2}^{N}$}}$&${\text{\small{32}}^{N}}$&${\text{\small{32}}^{\pm}}$&${\mathbf{\small{64}}^{N}}
		$&${\mathbf{\small 64}}^N_{\times 2}$&${\mathbf{\small{128}}^{N}} $&${\text{\small{128}}^{\pm}}$&$ {\text{\small{256}}^{N}} $
\\
\hline
		${\text{\small 5}} $&${\mathbf{\small 16}^{N}}$&${\text{\small 16}^{N}}$&${\text{\small 16}^{N}}$&${\text{\small 16}}^{\pm}$&${\mathbf{\small 32}}^{N}$&${\mathbf{\small 64}}^{N} $&${\mathbf{\small{128}}^{\pm}}$&${\text{\small{128}}^{N}} $&$\mathbf{\small{256}}^{N}$
\\
\hline\hline
		${\text{\small 4}} $&$  {\text{\small 8}^{\pm}}$&$ {\text{\small 8}^{\pm}}$&$
		{\text{\small 8}^{\pm}}$&$ 8_{\times 2}^{{\pm}}$&$16^{\pm}$&${\text{\small 32}}^{\pm}$&${\text{\small 64}}^{\pm}
		$&${\text{\small 64}_{\times 2}^{\pm}} $&${\text{\small{128}}^{\pm}}$
\\
\hline\hline
		${\text{\small 3}}$&${\mathbf{\small 8}}^{N}$&${\mathbf{\small 8}}^{N}$&${\mathbf{\small 8}^{N}}$&$8^{\pm}$&$16^{N}$&$32^{N}$
		&${\mathbf{\small 64}^{N}}$&$64^{\pm}$&${\mathbf{\small 128}}^{N}$
\\
\hline
		${\text{\small 2}}$&${\mathbf{\small 4}^{N}}$&$
		{\mathbf 4}_{\times 2}^{N}$&${\mathbf 8^{N}}$&$ 8^{\pm}$&$16^{N}$&$16_{\times 2}^{N}$&$32^{N}$&$32^{\pm} $&${\mathbf{\small 64}}^{N}$
\\
\hline
		${\text{\small 1}}$ &${\mathbf 2^{N}}$&${\mathbf 4^{N}}$&${\mathbf 8^{N}}$& $8^{\pm}$&$\mathbf {16}^{N}$&$16^{N}$&$16^{N}$&$16^{\pm}$&${\mathbf{\small 32}}^{N}$
\\
\hline\hline
		${\text{\small 0}} $&$  1^{\pm}$&$ 2^{\pm}$&$4^{\pm}$&$ 4_{\times 2}^{\pm}$&$ 8^{\pm}$&$ 8^{\pm}$&$ 8^{\pm}$&$ 8_{\times 2}^{\pm}$&$16^{\pm}$
		\\
		\hline\hline
		{s/r}&  {\text{\small 0}}& {\text{\small 1}}& 
		{\text{\small 2}}&{\text{\small 3}} & {\text{\small 4}}& {\text{\small 5}}& {\text{\small 6}}& {\text{\small 7}}& {\text{\small 8}}
		\\
		\hline
	\end{tabular}\label{t:dim}
\end{table}
We make the following comments to Table~\ref{t:dim}:
\begin{itemize}
\item[(1)] {We use the black colour  when $\dim(V^{r,s}_{\min})=2\dim(V^{r,s}_{irr})$, see Proposition~\ref{r-s = 0,1,2 mod 4}, 
items (2-2), (3-2-1), and (3-2-2).
}
\item[(2)] {Writing the subscript "$*_{\times 2}$" we show that the Clifford algebra has two minimal admissible modules corresponding to the non-equivalent irreducible modules, see Corollary~\ref{r-s = 3 mod 4 (2)}.
}
\item[(3)] {The upper index "$*^{N}$" means that the scalar product 
restricted to $E_{r,s}$ is neutral. 
The fact that $E_{r,s}$ is a neutral space does not depend on the
choice of the scalar product on $V^{r,s}_{min}$, see Section~\ref{sec:1eigenspace}. 
}
\item[(4)] { The upper index "$*^{\pm}$" shows that the scalar product
    restricted to the common 1-eigenspace $E_{r,s}$ of the system
    $PI_{r,s}$ is sign definite, see Section~\ref{sec:1eigenspace}. 
The sign of the scalar product on $E_{r,s}$ depends on the choice of the admissible scalar product on the module $V^{r,s}_{min}$.
}
\end{itemize}
For example, the Clifford algebra $\Cl_{3,0}$ has $4$ minimal
admissible modules $V^{3,0;\pm}_{min;\pm}$,
that is, each non-equivalent irreducible module $V^{3,0}_{irr;\pm}$ can be endowed 
with two scalar products: positive definite, giving the minimal admissible modules $V^{3,0;+}_{min;\pm}=V^{3,0;+}_{irr;\pm}$ and
negative definite: $V^{3,0;-}_{min;\pm}=V^{3,0;-}_{irr;\pm}$.
However, it does not mean that pseudo $H$-type Lie algebras 
corresponding to these four choices are different. We explain details in Theorems~\ref{main theorem 2} and \ref{main theorem 3}.
To obtain the Table~\ref{t:dim} 
we determine the sets $PI_{r,s}$ and $CO_{r,s}$, given in Section~\ref{app}.


\subsection{Scalar product on the common 1-eigenspace $E_{r,s}$}\label{sec:1eigenspace}


\begin{lemma}\label{lem:periodic}
Let $(V^{r,s}_{min}, \la\cdot\,,\cdot\ra_{V^{r,s}_{min}})$ 
be a minimal admissible module of $\Cl_{r,s}$ and 
$(V^{r+\mu,s+\nu}_{min}, \la\cdot\,,\cdot\ra_{V^{r+\mu,s+\nu}_{min}})$
a minimal admissible module of
$\Cl_{r+\mu,s+\nu}$, where $(\mu,\nu)\in\{(8,0)$, $(0,8),(4,4)\}$.
Let $E_{r,s}$  
and $E_{r+\mu,s+\nu}$ be the common 1-eigenspaces of the involutions $PI_{r,s}$ and $PI_{r+\mu,s+\nu}$, respectively.
Then $\dim E_{r,s}=\dim E_{r+\mu,s+\nu}$. Moreover,
if $E_{r,s}$ is a neutral vector space, then $E_{r+\mu,s+\nu}$ is also neutral, and 
if $E_{r,s}$ is a sign definite, then $E_{r+\mu,s+\nu}$ is also sign definite.
\end{lemma}

\begin{proof}
If $V_{min}^{r+\mu,s+\nu}=V_{min}^{r,s}\otimes V_{min}^{\mu,\nu}$, then 
the assertions follow from Proposition~\ref{periodicity by tensor product}.

Let $V_{min}^{r+\mu,s+\nu}$ be an arbitrary minimal admissible module
of $\Cl_{r+\mu,s+\nu}$, where $(\mu,\nu)\in\{ 
(8,0),(0,8),(4,4)\}$. Let $\{z_j,\zeta_{\alpha}\,;\,j=1,\ldots,{r+s},\,\alpha=1,\ldots, 8\}$ be
an orthonormal basis of
$\mathbb{R}^{r+\mu,s+\nu}$. We assume
$\{z_{i},\zeta_{\alpha}\,;\,i=1,\ldots,r,\,\alpha=1,\ldots,\mu\}$ 
are positive and
$\{z_{r+j},\zeta_{\mu+\beta}\,;\, j=1,\ldots,s,\,\beta=1,\ldots,\nu\}$ 
are negative. We identify $\mathbb{R}^{r,s}\oplus \mathbb{R}^{\mu,\nu}=\mathbb{R}^{r+\mu,s+\nu}$ by using the above bases. We choose $PI_{r+\mu,s+\nu}=PI_{r,s}\bigcup\{T_{i}\}_{i=1}^{4}$, where $T_i$ are involutions from Example~\ref{ex:inv44}.
The system of complementary operators $CO_{\mu,\nu}$
shows that the involutions $T_j\in PI_{r+\mu,s+\nu}$, $j=1,2,3,4$ decompose the space
$V^{r+\mu,s+\nu}_{min}$ into 16 common eigenspaces
$\{V_i\}_{i=0}^{15}$ of $T_j$ and
\begin{equation}\label{decomposition by Jzeta}
V^{r+\mu,s+\nu}_{min}=\bigoplus_{i=0}^{15} V_{i}=V_{0}
\bigoplus_{j=1}^{8} J_{\zeta_{j}}(V_{0})
\bigoplus_{j=2}^{8} J_{\zeta_{1}}J_{\zeta_{j}}(V_{0}),
\end{equation}
where $V_0$ is the common 1-eigenspace of $T_j$, $j=1,2,3,4$. Since the generators $J_{z_{j}}$, $j=1,\ldots,r+s$, commute with
involutions $T_{i}$, $i=1,2,3,4$, we can regard $V_{0}$ as a minimal
admissible module $V_{min}^{r,s}$ of $\Cl_{r,s}$. The involutions from $PI_{r,s}$ act on
$V_{0}=V_{min}^{r,s}$ and decompose it into their common 
eigenspaces. Then 
by definition $E_{r+\mu,s+\nu}=E_{r,s}$. This finishes the proof of the theorem.
\end{proof}

\begin{theorem}\label{th:common_sign}
Let $E_{r,s}\subset V^{r,s}_{min}$ be a common 1-eigenspace of the system $PI_{r,s}$. Then
the restriction of the admissible scalar product on $E_{r,s}$ is sign definite 
for $r\equiv 0,1,2\,(\text{\em mod}~4)$ and
$s\equiv 0 \,(\text{\em mod}~4)$ or for $r\equiv 3\,(\text{\em mod}~4)$ and
arbitrary $s$. 
Otherwise the restriction of the admissible scalar product on the common 1-eigenspaces $E_{r,s}$ is neutral.\end{theorem}

\begin{proof}
We find the sign of the scalar products on $E_{r,s}$ for basic cases~\eqref{basic cases} and then apply Lemma~\ref{lem:periodic}. 

{\sc Case $(r,0)$}. The scalar products on the common 1-eigenspaces $E_{r,0}$ are sign definite because the admissible scalar products on $V^{r,0}_{min}$ are sign definite.

{\sc Case $(r,4)$, $r=0,1,2,4$}. The system of involutions $PI_{r,4}$,
$r=0,1,2$ and their complementary 
operators are given in Table~\ref{tab:comr4}. The complementary operators gives the dimension of $E_{r,s}$ and the basis shows that the space is sign definite.
The case $(4,4)$ was considered in Example~\ref{ex:inv44}.

{\sc Cases $(3,s)$, $s=0,\ldots,7$ {\rm and} $(7,s)$, $s=1,2,3$.}  
The system of involutions and their complementary operators are given in Table~\ref{tab:com7s}. 
Notice that the involution $J_{z_1}J_{z_2}J_{z_3}$ belongs to all the
systems. The isometric complementary operators ensures that the common
1-eigenspace $E^{1}$ for involution from
$PI_{r,s}\setminus\{J_{z_1}J_{z_2}J_{z_3}\}$ 
is neutral. Let $E^{1,1}$ and $E^{1,-1}$ be the eigenspaces of $J_{z_1}J_{z_2}J_{z_3}$ corresponding to the eigenvalues 1 and -1, respectively. The last complementary operator from $CO_{r,s}$, that is anti-isometry, shows that the spaces $E^{1,1}\cap E^{1}$ and $E^{1,-1}\cap E^{1}$ are sign definite with opposite signs of scalar products on $E^{1,1}\cap E^{1}$ and $E^{1,-1}\cap E^{1}$. 

The case $E_{3,4}$ is special since there are two non-equivalent irreducible modules $V_{irr;+}^{3,4}$ and $V_{irr;-}^{3,4}$, where the volume form $\Omega^{3,4}=P_1P_3P_4$ acts as $\Id$ and $-\Id$ respectively. It shows that $P_4=J_{z_1}J_{z_2}J_{z_3}=-\Id$  on $V_{irr;+}^{3,4}$ and $P_4=J_{z_1}J_{z_2}J_{z_3}=\Id$  on $V_{irr;-}^{3,4}$.

The proof of the statement concerning the neutral common 1-eigenspace
follows from the systems $PI_{r,s}$ and $CO_{r,s}$ for mentioned
values of $r$ and $s$, 
see Table~\ref{tab:block3} in Section~\ref{app}.
\end{proof}


\section{Pseudo $H$-type Lie algebras and Lie groups}\label{sec:LieAlgebras}


In this section we recall basic facts on isomorphisms between pseudo
$H$-type algebras and discuss some properties of the automorphism
groups $\Aut(\n_{r,s}(U))$ of pseudo $H$-type algebras.
The Table~\ref{t:step1} contains the classification result for the pseudo $H$-type $\n_{r,s}(V^{r,s}_{min})$ obtained in~\cite{FM1}.


\subsection{Definitions of 
the pseudo $H$-type Lie algebras and their groups}\label{subsec:Defin.Pseudo}


Let $(U,\la\cdot\,,\cdot\ra_{U})$ be an admissible module of a Clifford
algebra $\Cl_{r,s}$.
We define a vector valued skew-symmetric bi-linear form 
\[
\begin{array}{cccccc}
[\cdot\,,\cdot]\colon &U\times U &\to &\mathbb{R}^{r,s}
\\
&(x,y)&\longmapsto &[x,y]
\end{array}
\]
by the relation
\begin{equation}\label{def of skew-symmetric bi-linear form}
\la J_{z}x,y\ra_{U}=\la z,[x,y]\ra_{r,s}.
\end{equation}

\begin{definition}{\em \cite{Ci}}
The space $U\oplus \mathbb{R}^{r,s}$ endowed with the Lie bracket
\[
[(x,z),(y,w)]=(0,[x,y])
\]
is called a pseudo $H$-type Lie algebra and it is denoted by $\n_{r,s}(U)$.
\end{definition}
A pseudo $H$-type Lie algebra $\n_{r,s}(U)$ is 2-step nilpotent, the space $\mathbb{R}^{r,s}$ is the centre, and the direct sum $U\oplus\mathbb{R}^{r,s}$ is orthogonal with respect to $\la\cdot\,,\cdot\ra_{U}+\la\cdot\,,\cdot\ra_{r,s}$.

The Baker-Campbell-Hausdorff formula allows us to define the Lie group structure on the space
$U\oplus\mathbb{R}^{r,s}$ by
\[
(x,z)*(y,w)=\bigr(x+y,z+w+\frac{1}{2}[x,y]\bigl).
\]
The Lie group is denoted by ${G}_{r,s}(U)$ and is called the pseudo $H$-type Lie group. 
Note that the scalar product $\la\cdot\,,\cdot\ra_{U}$ is implicitly
included in the definitions of the $H$-type Lie algebra and the
corresponding Lie group.
In general, the Lie algebra structure might change 
if we replace the admissible scalar product on $U$, see~\cite{AKMV, E, Eber03}. 
The main purpose of the present paper is to classify the Lie algebras $\n_{r,s}(U)$, whose constructions involve the non-minimal admissible modules $U$ of Clifford algebras~$\Cl_{r,s}$.


\subsection{Isomorphisms  of pseudo $H$-type Lie  algebras}\label{section Lie algebra isomorphism}


Let $U$ and $\tilde U$ 
be two vector spaces with scalar products $\la\cdot\,,\cdot\ra_{U}$  and $\la\cdot\,,\cdot\ra_{\tilde U}$ respectively. Let $\Lambda\colon U\to \tilde U$ be a linear map.
The operator $\Lambda^{\tau}\colon \tilde U\to U$ defined by the relation
\begin{equation}\label{adjoint}
\la \Lambda x,y\ra_{\tilde U}=\la x,\Lambda^{\tau}y\ra_{U}
\end{equation}
is called the adjoint operator with respect to the scalar products 
$\la\cdot,\cdot\ra_{U}$ and $\la\cdot,\cdot\ra_{\tilde{U}}$. If both scalar products are positive definite, we use the notation ${^t\Lambda}$.

Let $\n_{r,s}(U)$ and $\n_{\tilde{r},\tilde{s}}(\tilde{U})$
be two pseudo $H$-type Lie algebras with $r+s=\tilde r+\tilde s=k$ and $\dim(U)=\dim(\tilde U)=n$. A Lie algebra isomorphism $\Phi\colon \n_{r,s}(U)\to\n_{\tilde{r},\tilde{s}}(\tilde{U})$
has the form 
\[
\Phi=\begin{pmatrix}A&0\\B&C\end{pmatrix}\colon U\oplus\mathbb{R}^{r,s}
\longrightarrow \tilde{U}\oplus\mathbb{R}^{\tilde{r},\tilde{s}},\quad A\in \GL(n), \ C\in \GL(k),
\]
 $B\colon U\to\mathbb R^{k}$ is a linear map, see~\cite{KT}. The action is defined by
$
\Phi(x,z)=(Ax,Bx+Cz)$, $x\in U$, $z\in\mathbb R^{r,s}$.
If we write $J_z\colon U\to U$ and $\tilde J_w\colon\tilde U\to \tilde U$ for the corresponding actions on the Clifford modules, then the matrices $A$ and $C$ satisfy the relation
\begin{equation}\label{isomorphism relation}
A^{\tau}\tilde{J}_{w}A=J_{C^{\tau}(w)}\quad\text{for all}\quad  w\in\mathbb{R}^{\tilde{r},\tilde{s}},
\end{equation}
by~\eqref{def of skew-symmetric bi-linear form}. The matrices $A^{\tau}$ and $C^{\tau}$ are defined as in~\eqref{adjoint}.
The matrix $B$ is arbitrary and we can choose $B=0$ for
simplicity. 
To short the notation we write $\Phi=A\oplus C$ for the isomorphism $\Phi=\begin{pmatrix}A&0\\0&C\end{pmatrix}$. 
In following propositions we collect the properties of 
isomorphisms of $H$-type Lie algebras $\n_{r,s}(U)$ and $\n_{\tilde{r},\tilde{s}}(\tilde{U})$ for different values of signatures $(r,s)$ and $(\tilde{r},\tilde{s})$ studied in~\cite{FM1}, see also~\cite{Barbano,BTV,Riehm,Saal,TamYosh}.
\begin{proposition}
If $\Phi=A\oplus C\colon\n_{r,s}(U)\to\n_{\tilde{r},\tilde{s}}(\tilde{U})$
is a Lie algebra isomorphism, then
\begin{itemize}
\item[$(1)$] {the map $\Phi^{\tau}=A^{\tau}\oplus C^{\tau}\colon\n_{\tilde{r},\tilde{s}}(\tilde{U}) \to\n_{r,s}(U)$ is also a Lie algebra isomorphism and moreover}
\item[$(2)$] {$r=\tilde{r}$, $s=\tilde{s}$, or $r=\tilde{s}$, $s=\tilde{r}$.}
\end{itemize}
\end{proposition}
\begin{proposition}\label{Lie algebra isomorphism 1}   
If $\Phi=A\oplus C\colon\n_{r,s}(U)\to\n_{s,r}(\tilde{U})$ is a Lie algebra isomorphism and $r\not=s$, then 
\begin{itemize}
\item[$(1)$]{ $A^{\tau }A {J}_{z} A^{\tau}  A=-{J}_{z}$, \ 
 $z\in\mathbb R^{r,s}$, \quad $
AA^{\tau} \tilde{J}_{w} A 
A^{\tau}=-\tilde{J}_{w}$,  \ $w\in\mathbb R^{s,r}$;
}
\item[$(2)$]{
$AJ_{z_1}J_{z_2}=-\tilde{J}_{C(z_{1})}\tilde{J}_{C(z_{2})}A$,\quad
$A\tilde{J}_{C(z_1)}\tilde{J}_{C(z_2)}=-J_{z_1}J_{z_2} A^{\tau}$
for $z_1,z_2\in\mathbb{R}^{r,s}$ with $\la z_1,z_2\ra_{r,s}=0$;
}
\item[$(3)$]{
the linear transformation $C\colon \mathbb{R}^{r,s} \to\mathbb{R}^{s,r}$
maps positive vectors to negative vectors and vice versa. We can assume 
that $|\det A^{\tau} A|=1$ and $C^{\tau} C=-\Id$ by multiplying the matrix $A$ by a
constant.
}
\end{itemize}
\end{proposition}
\begin{proposition}\label{Lie algebra isomorphism 2}  
If $\Phi=A\oplus C\colon\n_{r,s}(U)\to\n_{r,s}(\tilde{U})$ is a Lie algebra isomorphism and $r\not=s$, then 
\begin{itemize}
\item[$(1)$]{
$A^{\tau}AJ_{z} A^{\tau} A={J}_{z}$,\qquad
$AA^{\tau}\tilde{J}_{w} AA^{\tau}=\tilde{J}_{w}$\quad 
for $z,w\in\mathbb R^{r,s}$;
}
\item[$(2)$]{
$AJ_{z_{1}}J_{z_{2}}=\tilde{J}_{C(z_{1})}\tilde{J}_{C(z_{2})}A$,\quad
$A\tilde{J}_{C(z_1)}\tilde{J}_{C(z_2)}=J_{z_{1}}J_{z_{2}}A^{\tau}$
for $z_{1},z_{2}\in\mathbb{R}^{r,s}$ with $\la z_{1},z_{2}\ra_{r,s}=0$;
}
\item[$(3)$]{the linear transformation $C\colon \mathbb{R}^{r,s} \to\mathbb{R}^{s,r}$
maps positive vectors to positive vectors, negative to negative ones, and as in Proposition~\ref{Lie algebra isomorphism 1}
we may assume that $|\det A^{\tau} A|=1$ and 
$C^{\tau}C=\Id$ by multiplying the matrix $A$ by a
constant.}
\end{itemize}
\end{proposition} 
\begin{proposition}\label{Lie algebra isomorphism 3}{\em \cite{FM1}}
If $\Phi=A\oplus C\colon\n_{r,r}(U)\to\n_{{r},{r}}(\tilde{U})$
is a Lie algebra isomorphism, then
$C^{\tau}C=\Id$ for $r\equiv 0,1,2\,(\text{\em mod}~4)$ 
and $C^{\tau}C=-\Id$ for $r\equiv 3\,(\text{\em mod}~4)$. 
The map $A$ can be normalised such that $|\det A^{\tau} A|=1$ and it satisfies the conditions of items $(1)-(2)$ of Proposition~\ref{Lie algebra isomorphism 2}. 
\end{proposition} 

We explain a possible construction of the map $A\colon V^{r,s}_{min}\to\tilde V^{r,s}_{min}$. It was shown in~\cite[Corollary 5, Theorem 3]{FM1} that the map $A$ can be obtained by the following procedure. The system of involutions $PI_{r,s}$ acting on the minimal admissible modules decomposes them into the direct sums $V^{r,s}_{min}=\oplus_i E_i$, $\tilde V^{r,s}_{min}=\oplus_i \tilde E_i$ of common eigenspaces. We start by constructing a map $A_1\colon E_{r,s}\to\tilde E_{\tilde r,\tilde s}$, where $E_{r,s}\subset V^{r,s}_{min}$, $\tilde E_{\tilde r,\tilde s}\subset\tilde V^{r,s}_{min}$ are common 1-eigenspaces of the system of involutions $PI_{r,s}$ and $PI_{\tilde r,\tilde s}$. Then the map $A_1$ produces the rest of the maps $A_i\colon E_i\to\tilde E_i$ between the eigenspaces. Thus the map $A$ has block diagonal form $A=\oplus_i A_i$ written in the basis described in Section~\ref{sec:lattice} and satisfying the relations of Propositions~\ref{Lie algebra isomorphism 1} and~\ref{Lie algebra isomorphism 2}.

\begin{example}\label{ex:2}
Recall decomposition~\eqref{eq:decom08} of $V^{\mu,\nu}_{min}$ 
for $(\mu,\nu)\in\{(8,0),(0,8),(4,4)\}$. Let $A\oplus \Id\colon n_{r,s}(V^{\mu,\nu}_{min})\to n_{r,s}(\tilde V^{\mu,\nu}_{min})$ and 
$$
\tilde V_{min}^{\mu,\nu}=\tilde E_{\mu,\nu}\bigoplus_{i=1}^{8}\tilde J_{\zeta_i}(\tilde E_{\mu,\nu})\bigoplus_{j=2}^{8}\tilde J_{\zeta_1}\tilde J_{\zeta_{j}}(\tilde E_{\mu.\nu})
$$
as in~\eqref{eq:decom08}. The condition~\eqref{isomorphism relation} applied to a vector $u\in E_{\mu,\nu}$ is equivalent to the statement that $A_{j+1}^{\tau}\tilde J_{\zeta_j}A_1=J_{\zeta_j}$, where $A_1=A\vert_{E_{\mu,\nu}}$ and $A_{j+1}^{\tau}=A^{\tau}\vert_{\tilde J_{\zeta_j}(\tilde E_{\mu,\nu})}$, $j=1,\ldots,8$, are the restrictions of the maps on the indicated spaces and the diagram 
\[
\begin{CD}
E_{\mu,\nu}@>>{J_{\zeta_{j}}}>J_{\zeta_{j}}(E_{\mu,\nu})
\\
@V{A_{1}}VV @AA{A_{j+1}^{\tau}}A
\\
\tilde E_{\mu,\nu}@>>{\tilde J_{\zeta_{j}}}>\tilde J_{\zeta_{j}}(\tilde E_{\mu,\nu})
\end{CD}
\]
commutes. This shows that the maps 
$$
A_{j+1}=\tilde J_{\zeta_j}(A_1^{-1})^{\tau}J_{\zeta_j}^{-1}\colon J_{\zeta_j}(E_{\mu,\nu})\to\tilde J_{\zeta_j}(\tilde E_{\mu,\nu}),\quad j=1,\ldots,8,
$$
are completely determined by the map $A_1$. The conditions $AJ_{\zeta_1}J_{\zeta_j}=\tilde J_{\zeta_1}\tilde J_{\zeta_j}A$ determine the maps 
$$
A_{j+8}=A_{1,j}=\tilde J_{\zeta_1}\tilde J_{\zeta_j}A_1(J_{\zeta_1}J_{\zeta_j})^{-1}\colon J_{\zeta_1}J_{\zeta_j}(E_{\mu,\nu})\to \tilde J_{\zeta_1}\tilde J_{\zeta_j}(\tilde E_{\mu,\nu}),\quad j=2,\ldots,8.
$$
Thus the map $A\colon V^{\mu,\nu}_{min}\to \tilde V^{\mu,\nu}_{min}$ is defined by $A_1\colon E_{\mu,\nu}\to \tilde E_{\mu,\nu}$ and has the form $A=\oplus_{j=1}^{j=16}A_j$ in a suitable basis.
\end{example}

\begin{lemma}\label{isomorphism between opposit sign scalar product}
Let $U^+=(U,\la\cdot\,,\cdot\ra_{U})$ be an admissible module of $\Cl_{r,s}$ and $J_z\colon U\to U$ be the action map, then the module  $U^-=(U,-\la\cdot\,,\cdot\ra_{U})$  
is an admissible with the same action map
$J_{z}\colon U^{-}=U\to U^{-}=U$. Moreover, 
the Lie algebras $\n_{r,s}(U)$ and $\n_{r,s}(\tilde{U})$ are
isomorphic under the isomorphism $\Id\oplus -\Id$.

\end{lemma}
\begin{proof} 
By the definition, we can see easily that $U^{-}=(U,-\la\cdot\,,\cdot\ra_{U})$ is
an admissible module and
the map 
\[
\begin{array}{cccc}
\Phi=\Id\oplus-\Id\colon &\n_{r,s}(U^{+})&\to&\n_{r,s}(U^{-})
\\
 &(x,z)&\mapsto &(x,-z)
\end{array}
\]
is a Lie algebra isomorphism because of
\begin{eqnarray*}
\la z,[x,y]\ra_{r,s}=\la J_{z}x,y\ra_{U^+}
=-\la J_{z}x,y\ra_{U^-}=\la J_{-z}x,y\ra_{U^-}
=
\la -z,[x,y]\ra_{r,s}.
\end{eqnarray*}
\end{proof}


\subsection{Automorphisms of the pseudo $H$-type Lie algebras}\label{Lie algebra automorphism}


We discuss here the group $\Aut(\n_{r,s}(U))$
of automorphisms of a Lie algebra $\n_{r,s}(U)$, see also~\cite{E,KT,Riehm}. The group $\Aut(\n_{r,s}(U))$ is 
a subgroup of $\GL(r+s+\dim(U),\mathbb{R})$ and consists of the linear maps 
$$
\Psi=\begin{pmatrix}A&0\\B&C\end{pmatrix}\colon \n_{r,s}(U)=U\oplus \mathbb{R}^{r,s}
\to U\oplus \mathbb{R}^{r,s}=\n_{r,s}(U)
$$
satisfying the condition~\eqref{isomorphism relation}, see also Propositions~\ref{Lie algebra isomorphism 2} and~\ref{Lie algebra isomorphism 3}.
The group $\Aut(\n_{r,s}(U))$
is isomorphic to the following product 
\begin{equation}\label{product structure of automorphism group}
\Aut(\n_{r,s}(u))\cong \mathbb{R}_{+}\times 
\bigr(B_{r,s}\rtimes \Aut^{0}(\n_{r,s}(U))\bigr).
\end{equation}
Here $\mathbb R_+$ is the group of non-homogeneous dilations $\delta_t\colon U\oplus \mathbb{R}^{r,s}
\to U\oplus \mathbb{R}^{r,s}$ acting as $\delta_t(x,z)=(tx,t^2z)$ for
$t\in \mathbb R_+$. 
The group
$B_{r,s}=\left\{\begin{pmatrix}\Id&0\\B&\Id\end{pmatrix}\right\}$ 
is isomorphic to $\mathbb{R}^{(r+s)\cdot \dim(U)}$. The subgroup $\Aut^{0}(\n_{r,s}(U))$,
consisting of the automorphisms of the form
$\Psi=\begin{pmatrix}A&0\\0&C\end{pmatrix}$ with
$C^{\tau}C=\pm \Id$, is called the group of {\it restricted automorphisms}. 
The semi-direct product in \eqref{product structure of automorphism group}
comes from the action of the subgroup $\Aut^{0}(\n_{r,s}(U))$
on 
$B_{r,s}$
by 
\[
\begin{pmatrix}A&0\\0&C\end{pmatrix}\cdot 
\begin{pmatrix}\Id&0\\B&\Id\end{pmatrix}\cdot
\begin{pmatrix}A^{-1}&0\\0&C^{-1}\end{pmatrix}
=\begin{pmatrix}\Id&0\\CBA^{-1}&\Id\end{pmatrix}\in B_{r,s}.
\]
The group of automorphisms of the Lie algebras
$\n_{r,0}(U)$ was studied 
in~\cite{Barbano,Kaplan81,Riehm,Saal}.

Now we present an example of elements of
$\Aut^{0}(\n_{r,s}(U))$, 
that will be important for the classification of the Lie algebras $\n_{r,s}(U)$. 
The map
$$
\mathbb{R}^{r,s}\ni z\mapsto -z\in \mathbb{R}^{r,s}\subset \Cl_{r,s}
$$ is extended to the Clifford algebra automorphism $\alpha\colon\Cl_{r,s} \to \Cl_{r,s}$
by the universal property of the Clifford algebras. 
We denote by $\Cl_{r,s}^{\times}$ the group of invertible elements in
$\Cl_{r,s}$ and in particular ${\mathbb{R}^{r,s}}^{\times}=\{v\in\mathbb R^{r,s}\vert\ \la v,v\ra_{r,s}\not=0\}
=\mathbb{R}^{r,s}\cap \Cl_{r,s}^{\times}$. The representation 
$\widetilde{\Ad}\colon{\mathbb R^{r,s}}^{\times}\to\End(\mathbb R^{r,s})$, 
is defined as
$$
\widetilde{\Ad}_{v}(z)=-v zv^{-1}
=\left(z-2\frac{\la z,v\ra_{r,s}}{\la v,v\ra_{r,s}}v\right)\in\mathbb R^{r,s}\quad 
\text{for}\quad z\in\mathbb R^{r,s},\ \ v\in {\mathbb{R}^{r,s}}^{\times}.
$$
Then it extends to the, so called, twisted adjoint representation 
$\widetilde{\Ad}\colon\Cl_{r,s}^{\times}\to\GL(\Cl_{r,s})$ by setting
\begin{equation}\label{twisted adjoint representation}
\Cl_{r,s}^{\times}\ni \varphi\longmapsto \widetilde{\Ad}_{\varphi},\quad
\widetilde{\Ad}_{\varphi}(\phi)=\alpha(\varphi) \phi \varphi^{-1}, \quad \phi\in \Cl_{r,s}.
\end{equation}
The map $\widetilde{\Ad}_{v}$ for $v\in{\mathbb{R}^{r,s}}^{\times}$, leaving the space
$\mathbb{R}^{r,s}\subset \Cl_{r,s}$ invariant, is also an isometry:
$
\la\widetilde{\Ad}_{v}(z),\widetilde{\Ad}_{v}(z)\ra_{r,s}=\la
z, z\ra_{r,s}$.
Note that ${(\widetilde{\Ad}_{\varphi^{-1}}})^{\tau}=\widetilde{\Ad}_{\varphi}$.
Subgroups of $\Cl_{r,s}^{\times}$ defined by
\begin{align*}
&\Pin(r,s)=\{v_{1}\cdots v_{k}\in \Cl_{r,s}^{\times}\mid\ \la v_{i},v_{i}\ra_{r,s}=\pm 1\},
\\
&\Spin(r,s)=\{v_{1}\cdots v_{k}\in \Cl_{r,s}^{\times}\mid\ k~\text{is even},\la v_{i},v_{i}\ra_{r,s}=\pm 1\},
\end{align*}
are called pin and spin groups, respectively. The reader can find more information about the twisted adjoint representation and the groups $\Pin$ and $\Spin$, see~\cite{LawMich}. 

\begin{proposition}\label{prop:coverings}{\em \cite{ABS, LawMich}}
The maps 
$$
\widetilde{\Ad}\colon\Pin(r,s)\to \Orth(r,s)\quad\text{and}\quad
\widetilde{\Ad}\colon\Spin(r,s) \to \SO(r,s) 
$$
are the double covering maps.
\end{proposition}

We make the identification $\Spin(r)\times \Pin(s)\cong
\Spin(r,0)\times \Pin(0,s)\subset \Pin(r,s)$ and present a special map
from $\Aut^{0}(\n_{r,s}(U))$.

\begin{proposition}\label{automorohisms from Clifford algebra} 
Let $J\colon\Cl_{r,s}\to\End(U)$ be a Clifford algebra representation and $\varphi\in \Spin(r)\times \Pin(s)$. Then
$J_{\varphi^{-1}}\oplus (\widetilde{\Ad}_{\varphi})^{\tau}
\in \Aut^{0}(\n_{r,s}(U))$.
The group homomorphism
$$
\begin{array}{ccccc}
\mathcal{A}\colon &\Spin(r)\times \Pin(s)&\to & \Aut^{0}(\n_{r,s}(U)),
\\
&\varphi&\mapsto &J_{\varphi^{-1}}\oplus (\widetilde{\Ad}_{\varphi})^{\tau}
\end{array}
$$
is injective and the diagram
\begin{equation}\label{CD11}
\begin{CD}
{0}@>>>K_{r,s}(U)@>>> \Aut^{0}(\n_{r,s}(U)) @>{A\oplus C \,\mapsto\, C}>> \Orth(r,s)@.\\
@.@AAA{\mathcal{A}}@AAA@AA{=}A\\
{0}@>>>\mathbb Z_{2}@>>> \Spin(r)\times \Pin(s) @ >{\widetilde{\Ad}}>>O(r,s)@.
\end{CD}
\end{equation}
is commutative. The kernel $K_{r,s}(U)$ consists of automorphisms of the
form
$A\oplus \Id$. 
\end{proposition}

\begin{proof}
By the definition of the twisted adjoint representation,
$
\alpha(\varphi)z \varphi^{-1}=\widetilde{\Ad}_{\varphi}(z)
$
we have
\begin{equation*}\label{candidate of automorphism}
J_{\alpha(\varphi)} J_{z}J_{\varphi^{-1}}=J_{\widetilde{\Ad}_{\varphi}(z)},\quad z\in{\mathbb R^{r,s}}^{\times}.
\end{equation*}
If we show that  $
J_{\alpha(\varphi)}=J_{\varphi^{-1}}^{\tau}$, or equivalently $J_{\alpha(\varphi^{-1})}=J_{\varphi}^{\tau}$
for $\varphi\in \Pin(r,s)$,  then it will imply that $J_{\varphi^{-1}}\oplus
{(\widetilde{\Ad}_{\varphi})}^{\tau} \in \Aut^{0}(\n_{r,s})$ due to the relation $A^{\tau}J_zA=J_{C^{\tau}(z)}$.

If $v\in\mathbb{R}^{r,s \times}$ is such that $\la v,v\ra_{r,s}=-1$, then
$J_{v^{-1}}^{\tau}=J_{v}^{\tau}=-J_{v}=J_{\alpha(v)}$, and hence
$J_{v^{-1}}\oplus (\widetilde{\Ad}_{v})^{\tau}\in\Aut^{0}(\n_{r,s}(U))$. If
$v$ is such that $\la v,v\ra_{r,s}=1$, then
$J_{v^{-1}}^{\tau}=J_{-v}^{\tau}=J_{v}\not=J_{\alpha(v)}$, and therefore 
the map $J_{v^{-1}}\oplus (\widetilde{\Ad}_{v})^{\tau}$ does not belong to $\Aut^{0}(\n_{r,s}(U))$.
If $\varphi=v_{1}v_{2}$ with $\la
v_{i},v_{i}\ra_{r,s}=\pm 1$, $i=1,2$, then
$J_{(v_{1}v_{2})^{-1}}=J_{v_{2}v_{1}}=J_{\alpha(v_{1}v_{2})}^{\tau}$. It implies
$
J_{(v_{1}v_{2})^{-1}}\oplus (\widetilde{\Ad}_{v_{1}v_{2}})^{\tau}\in \Aut^{0}(\n_{r,s}(U)).
$
In general, if $\varphi=x_{1}\cdots x_{2p}\cdot y_{1}\cdots y_{q}\in\Pin(r,s)$
with $\la x_i,x_i\ra_{r,s}=1$, $i=1,\ldots, 2p$, and $\la y_j,y_j\ra_{r,s}=-1$, $j=1,\ldots, q$, then we obtain 
\[
(J_{(x_{1}\cdots x_{2p}\cdot y_{1}\cdots y_{q})^{-1}})^{\tau}=
(J_{y_{q},\cdots y_{1}\cdot x_{2p}\cdots x_{1}})^{\tau}
=(-1)^{2p+q}J_{x_{1}\cdots x_{2p}\cdot y_{1}\cdots y_{q}}
=J_{\alpha(x_{1}\cdots x_{2p}\cdot y_{1}\cdots y_{q})}.
\]
\end{proof}

\begin{corollary}\label{-Id included}
 There is an automorphism
$A\oplus -\Id\in \Aut^{0}(\n_{2r,s}(U))$ for any $r,s$.
\end{corollary}
\begin{proof}
Observe that the image of the map $\Aut^{0}(\n_{2r,s})\to \Orth(2r,s)$
includes the group $\SO(2r)\times \Orth(s)$. This follows from diagram~\eqref{CD11} and Proposition~\ref{prop:coverings}. Since $-\Id\in \SO(2r)\times \Orth(s)$ belongs to the image of $\Aut^{0}(\n_{2r,s}(U))\to \Orth(2r,s)$ we conclude that there is $A\colon U\to U$ such that $A\oplus -\Id\in \Aut^{0}(\n_{2r,s}(U))$.
\end{proof}

Consider the diagram~\eqref{CD11} in special cases of $\n_{0,8}=\n_{0,8}(V^{0,8}_{min})$ and $\n_{8,0}=\n_{8,0}(V^{8,0}_{min})$. The following two diagrams are exact.
\[
\begin{CD}
\{0\}@>>> K_{0,8}\cong \mathbb{R}^{\times}@>>> \Aut^{0}(\n_{0,8})@>>> \Orth(0,8)\cong \Orth(8)@>>>\{\Id\}\\
@.@AAA@A{\mathcal{A}}AA@AA{=}A\\
\{0\}@>>>\mathbb{Z}_{2} @>>> \Pin(0,8)@>>> \Orth(0,8)\cong \Orth(8)@>>>\{\Id\}
\end{CD}
\]
\[
\begin{CD}
\{0\}@>>> K_{8,0}\cong \mathbb{R}^{\times}@>>> \Aut^{0}(\n_{8,0})@>>> \Orth(8)@>>>\{\Id\}\\
@.@AAA@A{\mathcal{A}}AA@AAA\\
\{0\}@>>>\mathbb{Z}_{2} @>>> \Spin(8)@>>>  \SO(8)@>>>\{\Id\}.
\end{CD}
\]
Since the Lie algebras $\n_{8,0}$ and $\n_{0,8}$ are isomorphic, see~\cite{FM1}, 
there exists an automorphism $A\oplus I_{1}\in
\Aut^{0}(\n_{8,0})$ which is not in the image
$\mathcal{A}(\Spin(8))$, 
where $I_{1}$ is defined as
$I_{1}(\zeta_{1})=-\zeta_{1}$, $I_{1}(\zeta_{j})=\zeta_{j}$ $j=2,\ldots,8$.

At the end of the section we formulate the relation between the existence of an automorphism and an isomorphism of a special type.

\begin{lemma}\label{lem:aut_iso}
A Lie algebra isomorphism $A\oplus \Id\colon \n_{r,s}(V_{min}^{r,s;+})\to\n_{r,s}(V_{min}^{r,s;-})$ exists if and only if there is a Lie algebra automorphism
$A\oplus -\Id\colon\n_{r,s}(V_{min}^{r,s;+})\to\n_{r,s}(V_{min}^{r,s;+})$.
\end{lemma}
\begin{proof}
Let us assume that $A\oplus \Id_{\mathbb R^{r,s}}\colon \n_{r,s}(V_{min}^{r,s;+})\to\n_{r,s}(V_{min}^{r,s;-})$, where 
$
A\colon V_{min}^{r,s;+} \to V_{min}^{r,s;-}$ is a Lie algebra isomorphism. We assume that the module actions on $V_{min}^{r,s;\pm}$ coincide, but the admissible scalar
products differ by the sign, that is $\la\cdot\,,\cdot\ra_{V_{min}^{r,s;+}}=-\la\cdot\,,\cdot\ra_{V_{min}^{r,s;-}}$. We denote the Lie brackets on the corresponding pseudo $H$-type Lie algebras
by $[x,y]^{\pm}$ for $x,y\in V_{min}^{r,s;\pm}$. Then
\begin{eqnarray*}
\la z,[x,y]^{+}\ra_{r,s}&=&\la z,[Ax,Ay]^{-}\ra_{r,s}=\la J_{z}Ax,Ay\ra_{V_{min}^{r,s;-}}
=-\la J_{z}Ax,Ay\ra_{V_{min}^{r,s;+}}
\\
&=&-\la z,[Ax,Ay]^{+}\ra_{r,s}.
\end{eqnarray*}
It shows that $A\oplus -\Id$ is an automorphism of $\n_{r,s}(V_{min}^{r,s;+})$. Now assuming that $A\oplus -\Id\in\Aut^0\big(\n_{r,s}(V_{min}^{r,s;+})\big)$, we obtain
\begin{eqnarray*}
\la z,[x,y]^{+}\ra_{r,s}&=&\la -z,[Ax,Ay]^{+}\ra_{r,s}=-\la J_{z}Ax,Ay\ra_{V_{min}^{r,s;+}}
=\la J_{z}Ax,Ay\ra_{V_{min}^{r,s;-}}
\\
&=&\la z,[Ax,Ay]^{-}\ra_{r,s}.
\end{eqnarray*}
Thus $A\oplus \Id\colon \n_{r,s}(V_{min}^{r,s;+})\to\n_{r,s}(V_{min}^{r,s;-})$ is an isomorphism. 
\end{proof}


\subsection{Existence of lattices on pseudo $H$-type Lie groups}\label{sec:lattice}


To achieve the full description of isomorphic Lie algebras $\n_{r,s}(U)$, where  
the admissible module $U$ is not necessarily minimal,
we need a special type of bases for the Clifford modules. These type of bases
also show the existence of lattices on the corresponding Lie groups, see~\cite{FM}.
It is enough to construct the bases only for minimal admissible modules, and then apply Proposition~\ref{periodicity by tensor product}.

Let $V^{r,s}_{min}$ be a minimal admissible module of 
the Clifford algebra $\Cl_{r,s}$ and $E_{r,s}$ be the common
1-eigenspace for the system $PI_{r,s}$ of involutions. 
We fix a vector $v\in E_{r,s}$ such that
$|\la v,v\ra_{V^{r,s}_{min}}|=1$. Then a basis $\{x_i\}_{i=1}^{N}$
of the module $V^{r,s}_{min}$ can be chosen by setting 
$$
x_{1}=v,\quad x_{2}=J_{z_1}v,\ldots, \quad x_{r+s}=J_{z_{r+s-1}}v,\ldots,\quad
x_{N}=J_{z_{i_{1}}}J_{z_{i_2}}\ldots J_{z_{i_{k}}}v,
$$
that is a subset of all the $2^{r+s}$ vectors obtained from $v$ by action of $J_{z_{i_1}}\ldots J_{z_{i_k}}$, 
$1\leq i_{i}<i_{2}<\cdots< i_{k}\leq r+s$.
The vector $v\in E_{r,s}$ can be picked up in such a way that
the basis in $E_{r,s}$ will be orthonormal due to the following proposition.
\begin{proposition}\cite[Lemma 2.9, Corollary 2.10]{FM}\label{orthogonal}
Let $(V,\la\cdot\,,\cdot\ra_V)$ be an admissible module, $\Lambda_1,\ldots,\Lambda_l$ symmetric linear operators on $V$ such that 
\begin{itemize}
\item[$(1)$] $\Lambda^2_k=-\Id_V$, $k=1,\ldots,l$;
\item[$(2)$] $\Lambda_k\Lambda_j=-\Lambda_j\Lambda_k$ for all $k,j=1,\ldots,l$.
\end{itemize}
Then for any $w\in V$ with $\la w,w\ra_V=1$ there is a vector $\tilde w$ satisfying
$
\la\tilde w,\Lambda_k\tilde w\ra_V=0$ and $\la\tilde w,\tilde w\ra_V=1$, for $k=1,\ldots,l$.
\end{proposition}
Since the involutions are symmetric all the eigenspaces are mutually orthogonal, that implies the orthonormality of the constructed basis. The construction of the basis also shows that
$
\la J_{z_i}x_{j}, x_{k}\ra_{V^{r,s}_{min}}=\la z_{i}, [x_{j}, x_{k}]\ra_{r,s}=\pm 1$ or $0$.
It follows that the structure constants of the Lie algebra $\n_{r,s}(U)$ are $\pm 1$ or 0. The concrete construction of bases for $\n_{r,s}(U)$ can be found in~\cite{FM}, see also~\cite{CD}. Applying the Mal\'cev criterion~\cite{Malcev}, we obtain the proposition.

\begin{proposition}\cite{Malcev}\label{existence of lattice}
Let $U$ be an admissible module of a Clifford algebra
$\Cl_{r,s}$. Then
there exists a lattice on the pseudo $H$-type Lie group $G_{r,s}(U)$. 
\end{proposition}

 
\subsection{Classification of pseudo $H$-type Lie algebras $\n_{r,s}(V_{min}^{r,s})$}


The classification of the pseudo $H$-type algebras $\n_{r,s}(V^{r,s}_{min})$, constructed from the minimal admissible modules was done in~\cite{FM1}. We summarise the results of the classification in Table~\ref{t:step1}.
\begin{table}[h]
\center\caption{Classification result after the first step}
\vskip0.3cm
\begin{tabular}{|c||c|c|c|c|c|c|c|c|c|}
\hline
$8$&$\cong$&&&&&&&&
\\
\hline
7&d&d&d&$\not\cong$&&&&&
\\
\hline
6&d&$\cong$&$\cong$&h&&&&&
\\
\hline
5&d&$\cong$&$\cong$&h&&&&&
\\
\hline
4&$\cong$&h&h&h&$\circlearrowright$&&&&\\\hline
3&d&$\not\cong$&$\not\cong$&$\not\circlearrowright$&d&d&d&$\not\cong$&\\\hline
2&$\cong$&h&$\circlearrowright$&$\not\cong$&d&$\cong$&$\cong$&h&$$\\\hline
1&$\cong$&$\circlearrowright$&d&$\not\cong$&d&$\cong$&$\cong$&h&$$\\\hline
0&&$\cong$&$\cong$&h&$\cong$&h&h&h&$\cong$\\\hline\hline
$s/r$&0&1&2&3&4&5&6&7&8\\
\hline
\end{tabular}\label{t:step1}
\end{table}
Here ``d'' stands for ``double'', meaning that $\dim V^{r,s}_{min}=2\dim
V^{s,r}_{min}$ and
``h'' (half) means that $\dim V^{r,s}_{min}=\frac{1}{2}\dim V^{s,r}_{min}$. 
The corresponding pairs are trivially non-isomorphic due to the 
different dimension of minimal admissible modules. 
The symbol $\cong$ denotes the Lie algebra $\n_{r,s}(V^{r,s}_{min})$ having 
isomorphic counterpart $\n_{s,r}(V^{s,r}_{min})$, the symbol $\not\cong$ shows that the Lie algebra $n_{r,s}(V^{r,s}_{min})$ is not  
isomorphic to $\n_{s,r}(V^{s,r}_{min})$. The notation $\circlearrowright$ indicates that 
the Lie algebra $\n_{r,r}(V^{r,r}_{min})$ admits automorphisms $A\oplus\Id$, 
and $\not\circlearrowright$ denotes the Lie algebra $\n_{r,r}(V^{r,r}_{min})$ that does not have this type of automorphism.


\section{Classification of pseudo $H$-type algebras}\label{sec:Classification}


In this section we state and prove the
classification of the pseudo $H$-type algebras $\n_{r,s}(U)$
with an arbitrary admissible modules $U$,
and fixed signature $(r,s)$. Eventually, the classification depends on the decomposition of $U$ on the minimal admissible modules. It is enough to classify basic cases~\eqref{basic cases} due to Theorem~\ref{th:periodic}. 
%
%
%
%


\subsection{Statements of main results on isomorphisms of Lie algebras $\n_{r,s}(U)$}\label{sec:DecomAdmMod}


In the rest of the paper we use the upper index $\pm$ to indicate
the scalar products that differ by sign: $V_{min}^{r,s;+}=(V_{min}^{r,s},\la\cdot\,,\cdot\ra_{V_{min}^{r,s}})$ 
and $V_{min}^{r,s;-}=(V_{min}^{r,s},-\la\cdot\,,\cdot\ra_{V_{min}^{r,s}})$. 
We also use the lower index $\pm$ 
to distinguish the minimal admissible modules, corresponding to non equivalent irreducible modules,
$V_{min;\pm}^{r,s;+}=(V_{min;\pm}^{r,s},\la\cdot\,,\cdot\ra_{V_{min;\pm}^{r,s}})$ 
and $V_{min;\pm}^{r,s;-}=(V_{min;\pm}^{r,s},-\la\cdot\,,\cdot\ra_{V_{min;\pm}^{r,s}})$ that were mentioned in Corollary~\ref{r-s = 3 mod 4 (2)}.  

Recall that Clifford modules are completely reducible, see
Proposition~\ref{prop:reducibility}
and any admissible module can be
decomposed into the orthogonal sum of minimal admissible modules, see
Proposition~\ref{properties of admissible module}, item~(2). To make the complete classification we decompose an admissible module $U$ of the Clifford algebra $\Cl_{r,s}$ into the direct sum of, possibly different, minimal admissible modules. We distinguish the following possibilities. 
\begin{itemize}
\item[$$] {If $r-s\not\equiv 3 \,(\text{mod}~4)$ and $s$ is arbitrary or $r-s\equiv 3 \,(\text{mod}~4)$ and $s$ is odd then 
\begin{equation}\label{eq:rsneq3}
U= \big(\stackrel{p^+}\oplus V_{min}^{r,s;+}\big)\bigoplus\big(\stackrel{p^-}\oplus V_{min}^{r,s;-}\big).
\end{equation}
}
\item[$$] {If $r-s\equiv 3\,(\text{mod}~4)$ and $s$ is even, then 
\begin{equation}\label{eq:rseq3}
U= \big(\stackrel{p^+_+}\oplus V_{min;+}^{r,s;+}\big)\bigoplus \big(\stackrel{p^-_+}\oplus V_{min;+}^{r,s;-}\big)
\bigoplus
\big(\stackrel{p^+_-}\oplus V_{min;-}^{r,s;+}\big)\bigoplus \big(\stackrel{p^-_-}\oplus V_{min;-}^{r,s;-}\big).
\end{equation}
}
\end{itemize} 
The system of involutions $PI_{r,s}$ does not depend on the scalar product on the admissible modules $V_{min}^{r,s;\pm}$ and therefore the common 1-eigenspaces $E_{r,s}$ coincide on $V_{min}^{r,s;+}$ and $V_{min}^{r,s;-}$. Nevertheless, the restrictions of the admissible scalar products from $V_{min}^{r,s;+}$ and $V_{min}^{r,s;-}$ on the respective $E_{r,s}$ will have opposite signs. The result of the classification essentially depends on the signature of the restriction of the admissible scalar product on $E_{r,s}$ and the parity of the index $s$. We formulate the main results of the classification.

\begin{theorem}\label{main theorem 1} Let $U=(U,\la\cdot\,,\cdot\ra_{U})$ and
$\tilde U=(\tilde{U},\la\cdot\,,\cdot\ra_{\tilde{U}})$
be admissible modules of a Clifford algebra $\Cl_{r,s}$.
If $r\equiv 0,1,2\,\text{mod}~4$, 
then the pseudo $H$-type Lie algebra $\n_{r,s}(U)$ 
is determined by the dimension
of the admissible module $U$ and does not depend on the choice of an admissible scalar product. Thus $\n_{r,s}(U)\cong \n_{r,s}(\tilde U)$, if and only if $\dim(U)=\dim(\tilde U)$.
\end{theorem}

If $r\equiv 3\,(\text{mod}~4)$, then the pseudo $H$-type Lie algebra $\n_{r,s}(U)$  
is determined by the dimension of $U$ and by the value of the index $s$.

\begin{theorem}\label{main theorem 2}
Let $U=(U,\la\cdot\,,\cdot\ra_{U})$ and
$\tilde U=(\tilde{U},\la\cdot\,,\cdot\ra_{\tilde{U}})$
be admissible modules of a Clifford algebra $\Cl_{r,s}$.
Let $r\equiv 3\,(\text{\em mod}~4)$ and $s\equiv 0\,(\text{\em mod}~4)$ 
and let the admissible modules be decomposed into the direct sums: 
\begin{align*}
&U= (\stackrel{p^+_+}\oplus V_{min;+}^{r,s;+})\bigoplus(\stackrel{p^{-}_+}\oplus V_{min;+}^{r,s;-})
\bigoplus (\stackrel{p^+_-}\oplus V_{min;-}^{r,s;+})\bigoplus(\stackrel{p^{-}_-}\oplus V_{min;-}^{r,s;-}),
\\
&\tilde{U}=(\stackrel{\tilde{p}^+_+}\oplus V_{min;+}^{r,s;+})\bigoplus(\stackrel{\tilde{p}^{-}_+}\oplus V_{min;+}^{r,s;-})
\bigoplus (\stackrel{\tilde{p}^+_-}\oplus V_{min;-}^{r,s;+})\bigoplus(\stackrel{\tilde{p}^{-}_-}\oplus V_{min;-}^{r,s;-}).
\end{align*}
Then the Lie algebras $\n_{r,s}(U)$ and $\n_{r,s}(\tilde{U})$ are isomorphic, if and only if,
\[
p=p^+_++p^{-}_-=\tilde{p}^+_++\tilde{p}^{-}_-=\tilde p\ \ \text{and}\ \ q=p^{-}_++ p^+_-=\tilde{p}^{-}_++\tilde{p}^+_-=\tilde q
\]
or 
\[
p=p^+_++p^{-}_-=\tilde{p}^{-}_++\tilde{p}^+_-=\tilde q\ \ \text{and}\ \ q=p^{-}_++ p^+_-=\tilde{p}^+_++\tilde{p}^{-}_-=\tilde p.
\]
\end{theorem}

\begin{theorem}\label{main theorem 3}
Let $r\equiv 3\,(\text{\em mod}~4)$ and 
$s\equiv 1,2,3\,(\text{\em mod}~4)$ 
and let $U$ and $\tilde U$ be decomposed into the direct sums 
$$
U= (\stackrel{p^+}\oplus V_{min}^{r,s;+})\bigoplus (\stackrel{p^-}\oplus V_{min}^{r,s;-}),\qquad
\tilde{U}=(\stackrel{\tilde{p}^+}\oplus V_{min}^{r,s;+})\bigoplus(\stackrel{\tilde{p}^-}\oplus V_{min}^{r,s;-}).
$$
Then $\n_{r,s}(U)\cong \n_{r,s}(\tilde{U})$, if and only if
$p=p^+=\tilde{p}^+=\tilde p$ and $q=p^-=\tilde{p}^-=\tilde q$, or $p=p^+=\tilde{p}^-=\tilde q$ and $q=p^-=\tilde{p}^+=\tilde p$. 
\end{theorem}


\subsection{Periodicity of isomorphisms}\label{periodicity of iso}


We can reduce the proof of the main theorems to basic cases~\eqref{basic cases}, due to the following facts. Let $V_{min}^{\mu,\nu}$ be a minimal admissible module of the
Clifford algebra $\Cl_{\mu,\nu}$, where $(\mu,\nu)\in
\{(8,0),(0,8),(4,4)\}$. It was explained in Example~\ref{ex:inv44} that
$V_{min}^{\mu,\nu}$ admits decomposition~\eqref{eq:decom08}. The admissible scalar product restricted to $E_{\mu,\nu}$ is
necessarily sign definite and we can fix it to be positive
definite scalar product on $E_{\mu,\nu}$ by Lemma~\ref{isomorphism between opposit sign scalar product}. 
We summarise the results of Section~\ref{sec:Bott} and Lemma~\ref{lem:periodic}.

\begin{proposition}\label{periodicity all}
Let $(V^{r,s}_{min},\la\cdot\,,\cdot\ra_{V^{r,s}_{min}})$ 
be a minimal admissible module of $\Cl_{r,s}$ and
$J_{z_i}$, $i=1,\ldots,r+s$ 
the Clifford actions of the orthonormal basis $\{z_{i}\}$. Then
\begin{equation}\label{tensor product rep to direct sum}
V^{r,s}_{min}\otimes V_{min}^{\mu,\nu}
=(V^{r,s}_{min}\otimes E_{\mu,\nu})\bigoplus_{i=1}^{8}\big(V^{r,s}_{min}\otimes J_{\zeta_i}(E_{\mu,\nu})\big)\bigoplus_{j=2}^{8}\big(V^{r,s}_{min}\otimes J_{\zeta_1}J_{\zeta_{j}}(E_{\mu,\nu})\big)
\end{equation}
is a minimal admissible module $V^{r+\mu,s+\nu}_{min}$ of the Clifford algebra $\Cl_{r+\mu,s+\nu}$.

Conversely, if $V^{r+\mu,s+\nu}_{min}$ is a minimal admissible module of the algebra
$\Cl_{r+\mu,s+\nu}$, then the common 1-eigenspace $E_0$ of the involutions
$T_{i}$, $i=1,2,3,4$ from Example~\ref{ex:inv44} can be considered as a minimal admissible module $V^{r,s}_{min}$ of the algebra
$\Cl_{r,s}$. The action of the Clifford algebra $\Cl_{r,s}$ on
$E_{0}$ is the restricted action of $\Cl_{r+\mu,s+\nu}$ obtained by the natural inclusion 
$\Cl_{r,s}\subset \Cl_{r+\mu,s+\nu}$. 
\end{proposition}

\begin{proposition}\label{extension of automorphism}
According to the correspondence of minimal admissible modules 
stated in Proposition~\ref{periodicity all}, there
is a natural injective map
\[
\mathcal{B}\colon\Aut^{0}(\n_{r,s}(V^{r,s}_{min}))\to
\Aut^{0}(\n_{r+\mu,s+\nu}(V^{r+\mu,s+\nu}_{min})).
\]

Conversely, automorphisms of the form ${A}\oplus
C\in \Aut^{0}(\n_{r+\mu,s+\nu}(V^{r+\mu,s+\nu}_{min}))$ with the property that
$C(\zeta_{j})=\zeta_{j}$, $j=1,\ldots,8$,
defines an automorphism $A_{|E_{0}}\oplus C_{|\mathbb{R}^{r,s}}$ of 
the algebra $\n_{r,s}(E_0)$, where
the space $E_0$ is the common 1-eigenspace of the 
involutions $T_{j}$, $j=1, 2, 3, 4$, viewed as a minimal admissible module of $\Cl_{r,s}$.
\end{proposition}

\begin{proof}
Let $A^{r,s}\oplus C\in\Aut^0(\n_{r,s}(V^{r,s}_{min}))$ with $(A^{r,s})^{\tau}J_{z_i}A^{r,s}=J_{C^{\tau}(z_i)}$, $i=1,\ldots,r+s$,
and let $J_{\zeta_j}$, $j=1,\ldots,8$, be the actions on
$V_{min}^{\mu,\nu}$ of the Clifford algebra $\Cl_{\mu,\nu}$.

We want to construct $\bar A\colon
V^{r+\mu,s+\nu}_{min}=V_{min}^{r,s}\otimes V_{min}^{\mu,\nu}\to
V^{r+\mu,s+\nu}_{min}=V_{min}^{r,s}\otimes V_{min}^{\mu,\nu}$ 
such that $z_i\mapsto C(z_i)$ and $\zeta_j\mapsto \zeta_j$  by using
the map $A^{r,s}\colon V^{r,s}_{min}\to V^{r,s}_{min}$. 

The action $\bar{J}$ on 
$V_{min}^{r,s}\otimes V_{min}^{\mu,\nu}$ 
defined in 
Proposition~\ref{periodicity by tensor product}
corresponds to
\[\left\{ 
\begin{array}{ll}
\bar{J}_{z_i}(x\otimes y)
=-J_{z_i}(x)\otimes y, &\quad x\in V^{r,s}_{min},\ \ y\in J_{\zeta_{j}}(E_{0}),\, j=1,\ldots,8,\\
\bar {J}_{z_i}(x\otimes y)
=J_{z_i}(x)\otimes y,&\quad x\in V^{r,s}_{min},\ \ y\in J_{\zeta_{1}}J_{\zeta_{j}}(E_{0}),\, j=2,\ldots,8,\\
\bar {J}_{\zeta_{i}}(x\otimes y)=x\otimes J_{\zeta_{j}}(y), &\quad x\in
                                                            V^{r,s}_{min},\
                                                            \ y\in E_{0},
\end{array}\right.
\]
according to the decomposition~\eqref{tensor product rep to direct sum}.
We define $\bar{A}\colon V_{min}^{r,s}\otimes V_{min}^{\mu,\nu}\to V_{min}^{r,s}\otimes V_{min}^{\mu,\nu}$
on each component of the decomposition~\eqref{tensor product rep to direct sum}
such that it satisfies~\eqref{isomorphism relation} with
$\bar{C}$ being $\bar{C}(z_{i})=C(z_{i})$ and $\bar{C}(\zeta_{j})=\zeta_{j}$.
It can be done in a unique way as in Example~\ref{ex:2} and the operator $\bar{A}\oplus\bar{C}$ will satisfy Proposition~\ref{Lie algebra isomorphism 2}.

Conversely, let $\bar A\oplus\bar C\in
\Aut^{0}(\n_{r+\mu,s+\nu}(V^{r+\mu,s+\nu}_{min}))$ 
be such that $\bar C(\zeta_{j})=\zeta_{j}$, $j=1,\ldots,8$.
Then $V^{r+\mu,s+\nu}_{min}$ is decomposed into the orthogonal 
sum~\eqref{decomposition by Jzeta} 
and the commutativity of the
operators $J_{z_i}$ with the involutions $T_{j}$ allows us to define an
automorphism $A_{|E_{0}}\oplus C_{|\mathbb{R}^{r,s}}$ of
the pseudo $H$-type algebra $\n_{r,s}(E_{0})$.
\end{proof}

Note that the construction given in Proposition~\ref{periodicity by tensor product} 
can be performed for an arbitrary, not necessary minimal admissible module $U^{r,s}$. Thus  we obtain that $U^{r+\mu,s+\nu}=U^{r,s}\otimes V^{\mu,\nu}_{min}$ is admissible for $(\mu,\nu)\in\{(8,0),(0,8),(4,4)\}$ if $U^{r,s}$ is admissible. Denote by $K_{r,s}(U^{r,s})$ the kernel of the map 
$$
\begin{array}{ccccc}
&\Aut^{0}(\n_{r,s}(U^{r,s})) &\longrightarrow &\Orth(r,s)
\\
& A\oplus C&\mapsto &C.
\end{array}
$$
\begin{corollary}
Let $U^{r,s}$ and $U^{r+\mu,s+\nu}=U^{r,s}\otimes V^{\mu,\nu}_{min}$ be admissible modules. 
Then 
\[
K_{r+\mu,s+\nu}(U^{r+\mu,s+\nu})=\mathcal{B}(K_{r,s}(U^{r,s})),
\]
that is the kernel $K_{r,s}(U^{r,s})$ is invariant under the map $\mathcal{B}$ defined in Proposition~\ref{extension of automorphism}.
\end{corollary}

\begin{lemma}\label{lem:extisom}
If the Lie algebras $\n_{r,s}(V^{r,s}_{min})$ and  $\n_{r,s}(\tilde V^{r,s}_{min})$ are isomorphic, then the Lie algebras $\n_{r+\mu,s+\nu}(V^{r+\mu,s+\nu}_{min})$ and $\n_{r+\mu,s+\nu}(\tilde V^{r+\mu,s+\nu}_{min})$ are also isomorphic for $(\mu,\nu)\in \{(8,0),(0,8),(4,4)\}$.
\end{lemma}
\begin{proof}
Let $\Phi=A\oplus C\colon \n_{r,s}(V^{r,s}_{min})\to\n_{r,s}(\tilde V^{r,s}_{min})$ be a Lie algebra isomorphism, then $\Phi$ that can be extended to $\bar \Phi=\bar A\oplus \bar C\colon \n_{r+\mu,s+\nu}(V^{r+\mu,s+\nu}_{min})\to\n_{r+\mu,s+\nu}(\tilde V^{r+\mu,s+\nu}_{min})$ by the same procedure as in Lemma~\ref{extension of automorphism}. Namely, we set $\bar A=A\otimes \Id$ and $\bar C=C\oplus \Id$.
\end{proof}

Let now $U^{r,s}$ and $\tilde U^{r,s}$ be two admissible modules of equal dimensions for the Clifford algebra $\Cl_{r,s}$. Then $U^{r,s}=\oplus_k(V^{r,s}_{min})_k$ and $\tilde U^{r,s}=\oplus_k(\tilde V^{r,s}_{min})_k$. Then admissible modules $U^{r+\mu,s+\nu}$ and $\tilde U^{r+\mu,s+\nu}$ can be identified with 
$$
U^{r+\mu,s+\nu}\sim U^{r,s}\otimes V^{\mu,\nu}_{min}=\oplus_k \big((V^{r,s}_{min})_k\otimes V^{\mu,\nu}_{\min}\big)\sim \oplus_k(V^{r+\mu,s+\nu}_{min})_k,
$$
and 
$$
\tilde U^{r+\mu,s+\nu}\sim \tilde U^{r,s}\otimes V^{\mu,\nu}_{min}=\oplus_k \big((\tilde V^{r,s}_{min})_k\otimes V^{\mu,\nu}_{\min}\big)\sim \oplus_k(\tilde V^{r+\mu,s+\nu}_{min})_k.
$$
Now applying Lemma~\ref{lem:extisom} we obtain the following result.

\begin{theorem}\label{th:periodic}
If the Lie algebras $\n_{r,s}(U^{r,s})$ and  $\n_{r,s}(\tilde U^{r,s})$ are isomorphic, then the Lie algebras $\n_{r+\mu,s+\nu}(U^{r+\mu,s+\nu})$ and $\n_{r+\mu,s+\nu}(\tilde U^{r+\mu,s+\nu})$ are also isomorphic for $(\mu,\nu)\in \{(8,0),(0,8),(4,4)\}$.
\end{theorem}
 

\subsection{Proof of Theorem~\ref{main theorem 1}}


In order to prove the classification theorems for the pseudo $H$-type Lie algebras $\n_{r,s}(U)$,  one should be careful about the scalar product 
on each minimal admissible component of the decompositions~\eqref{eq:rsneq3} and~\eqref{eq:rseq3} of the admissible module $U$. Let us assume that $U=\oplus_i V_i$, where $V_i$ are minimal admissible modules. If we find linear maps $A_{ij}\colon V_i\to V_j$ for all $i$ and $j$ such that $A_{ij}\oplus C\colon\n_{r,s}(V_i)\to\n_{r,s}(V_j)$ are the Lie algebra isomorphisms with $CC^{\tau}=\Id_{\mathbb R^{r,s}}$, then the Lie algebra $\n_{r,s}(U)$ is unique. Even though a different choice of a scalar product on the vector space $U$ gives different minimal admissible modules $V_i$ in the decomposition $U=\oplus_i V_i$, the resulting Lie algebras $\n_{r,s}(U)$ can be isomorphic if there is a map $C\colon\mathbb R^{r,s}\to\mathbb R^{r,s}$ that is the same for all $A_{ij}\colon V_i\to V_j$. For the simplicity we choose $C$ to be identity on $\mathbb R^{r,s}$. 
The construction of maps $A_{ij}\colon V_i\to V_j$ depends on the signature of the restriction of the admissible scalar product on the common $1$-eigenspace of each minimal admissible module $V_i$ from the direct sum $U=\oplus_i V_i$. The proof of Theorem~\ref{main theorem 1} is given in three lemmas according to whether the common $1$-eigenspace is sign definite or neutral space and depends also on the type of the decomposition $U=\oplus_i V_i$ in~\eqref{eq:rsneq3} and~\eqref{eq:rseq3}. The first lemma concerns with the cases when there are only two types of minimal admissible modules $V_{min}^{r,s;+}$ and $V_{min}^{r,s;-}$ and 
the restrictions of the scalar product onto the common 1-eigenspace is sign definite and have different sign. 
 
\begin{lemma}\label{lem:isom411}
The Lie algebras $\n_{r,s}(V_{min}^{r,s;+})$ and
$\n_{r,s}(V_{min}^{r,s;-})$ are 
isomorphic for $r\equiv 0,1,2\,(\text{\em mod}~4)$ $s\equiv
0\,(\text{\em mod}~4)$ 
under the map $A\oplus\Id_{\mathbb R^{r,s}}$.
\end{lemma}
\begin{proof}
We consider the case $r\equiv 0,2\,(\text{mod}~4)$. 
In this case the existence of an
isomorphism $A\oplus\Id_{\mathbb
  R^{r,s}}\colon\n_{r,s}(V_{min}^{r,s;+})\to
\n_{r,s}(V_{min}^{r,s;-})$ is equivalent to the existence of the
automorphism $A\oplus-\Id_{\mathbb R^{r,s}}\colon
\n_{r,s}(V_{min}^{r,s;+})\to \n_{r,s}(V_{min}^{r,s;+})$ by
Lemma~\ref{lem:aut_iso}. The necessary automorphism exists by
Corollary~\ref{-Id included}. 

Let $r\equiv 1\,(\text{mod}~4)$. 
We need only to consider the cases $(1,0)$ $(5,0)$ and $(1,4)$ due to periodicity. The case $\n_{1,0}(V_{min}^{1,0;\pm})$ is trivial by the uniqueness of three dimensional Heisenberg algebra.
 
Let $(r,s)=(5,0)$. We can assume that the minimal admissible module $V^{6,0}_{min}$ carries positive definite scalar product due to the first part of Lemma~\ref{lem:isom411}. Moreover, since $\dim(V^{6,0}_{min})=\dim(V^{5,0;+}_{min})=8$ the minimal admissible module $V_{min}^{5,0;+}$ can be thought as the
restriction of 
the minimal admissible module
$V_{min}^{6,0}\cong\mathbb{R}^{8,0}$ 
by restricting the action of $J_z\colon \mathbb R^{6,0}\to\End(V_{min}^{6,0})$ through the inclusion map
$\mathbb{R}^{5,0}\subset\mathbb{R}^{6,0}$ as well as the restriction of the scalar product $\la\cdot\,,\cdot\ra_{V^{6,0}_{min}}$ onto $V^{5,0;+}_{min}$. Let $\pi\colon\mathbb{R}^{6,0}\to \mathbb{R}^{5,0}$ be the orthogonal projection. Then  
\begin{equation}\label{eq:6050}
\Id_{\mathbb R^{8,0}}\oplus\,\pi\colon\n_{6,0}(V_{min}^{6,0})\to \n_{5,0}(V_{min}^{6,0})
\end{equation}
is a Lie algebra homomorphism. Let $A\oplus -\Id_{\mathbb R^{6,0}} \in \Aut^{0}(\n_{6,0}(V^{6,0}_{min}))$, which existence is guaranteed by Corollary~\ref{-Id included}. The property
$$
A^{\tau} J_{z} A=J_{-z}\quad\text{for any}\quad z\in \mathbb{R}^{6,0},\quad
~\text{with}
~\la Ax, y\ra_{V_{min}^{6,0}}=\la x, A^{\tau}y\ra_{V_{min}^{6,0}}
$$
and the homomorphism~\eqref{eq:6050} allow to descend the automorphism $A\oplus -\Id\in \Aut^{0}(\n_{6,0}(V_{min}^{6,0}))$ to the 
automorphism of $\n_{5,0}(V_{min}^{5,0;+})$ with the same map $A\colon V_{min}^{5,0;+}\to V_{min}^{5,0;+}$. Now applying Lemma~\ref{lem:aut_iso}, we conclude that there is an isomorphism $A\oplus\Id\colon V_{min}^{5,0;+}\to V_{min}^{5,0;-}$ that finishes the proof.

The existence of an isomorphism $A\oplus\Id\colon V_{min}^{1,4;+}\to V_{min}^{1,4;-}$ can be deduced from the existence of an automorphism $A\oplus-\Id\colon V_{min}^{2,4;+}\to V_{min}^{2,4;+}$ as in the previous case.
\end{proof} 

The following two lemmas give the rest of the proof of
Theorem~~\ref{main theorem 1} and they are concerned with indices
$r\equiv 0,1,2\,(\text{mod}~4)$, $s\equiv 1,2,3\,(\text{mod}~4)$ for 
which the restriction of the admissible scalar product to the common 1-eigenspace is neutral. Lemma~\ref{uniqueness(0)-2} deals with the indices $(r,s)\notin\{(1,2),(1,6),(5,2)\}$, because in these cases any admissible module $U$ has decomposition~\eqref{eq:rsneq3}. If $(r,s)$ belongs to $\{(1,2),(1,6),(5,2)\}$ then an admissible module $U$ has decomposition~\eqref{eq:rseq3} and the results of Lemma~\ref{uniqueness(0)-2} are extended in Lemma~\ref{lem:121652}.

\begin{lemma}\label{uniqueness(0)-2}
Let $(V_{min}^{r,s;+}=(V,\la\cdot\,,\cdot\ra_{V_{min}^{r,s;+}})$ and $V_{min}^{r,s;-}=(V,\la\cdot\,,\cdot\ra_{V_{min}^{r,s;-}})$ be two minimal admissible modules of $\Cl_{r,s}$ with $\la\cdot\,,\cdot\ra_{V_{min}^{r,s;+}}=-\la\cdot\,,\cdot\ra_{V_{min}^{r,s;-}}$. If the restrictions of both scalar products on the common 1-eigenspace $E_{r,s}$ of involutions from $PI_{r,s}$ 
are neutral, then the Lie algebras $\n_{r,s}(V_{min}^{r,s;+})$ and $\n_{r,s}(V_{min}^{r,s;-})$ are isomorphic under the isomorphism $A\oplus \Id_{\mathbb R^{r,s}}$.
\end{lemma}
\begin{proof}
Notice that the system of involutions $PI_{r,s}$ does not depend on the scalar product and therefore the common 1-eigenspace $E_{r,s}$ is the same for both modules. The restrictions of the scalar products on $E_{r,s}$ are neutral by hypothesis. We find $v,u\in E_{r,s}$ such that $\la v,v\ra_{V_{min}^{r,s;+}}=\la u,u\ra_{V_{min}^{r,s;-}}=1$. We find the orthonormal bases
$$
x_0=v_1,\ x_1=J_{z_1}v,\ \ldots, x_{N-1}=\prod_i J_{z_i}v,\quad N=\dim(V),
$$
$$
y_0=u,\ y_1=J_{z_1}u,\ \ldots, y_{N-1}=\prod_i J_{z_i}u,\quad N=\dim(\tilde V).
$$
Then the map $A\oplus \Id_{\mathbb R^{r,s}}$ is the isomorphism of the Lie algebras $\n_{r,s}(V_{min}^{r,s;+})$ and $\n_{r,s}(V_{min}^{r,s;-})$, where we set
$
A\colon x_i\mapsto y_i$ and then extended it by linearity.
Indeed, let $z_k\in\mathbb R^{r,s}$ be arbitrary and $J_{z_k}\colon V\to V$ be the Clifford action. Then
\begin{eqnarray*}
\la z_k,[x_i,x_j]\ra_{r,s}& = &
\la J_{z_k}\prod_i J_{z_{k_i}}v,\prod_j J_{z_{k_j}}v\ra_{V_{min}^{r,s;+}}=
\la J_{z_k}\prod_i J_{z_{k_i}}u,\prod_j J_{z_{k_j}}u\ra_{V_{min}^{r,s;-}}
\\
&&
\la z_k,[y_i,y_j]\ra_{r,s}
\end{eqnarray*}
since the calculations depend only on number of permutations in the products.
\end{proof}

In the cases $(r,s)\in\{(1,2),(1,6),(5,2)\}$ there are two irreducible modules $V_{irr;\pm}^{r,s}$, but they are not admissible. The minimal admissible modules are $V_{min;+}^{r,s}=V_{irr;+}^{r,s}\oplus V_{irr;+}^{r,s}$ and $V_{min;-}^{r,s}=V_{irr;-}^{r,s}\oplus V_{irr;-}^{r,s}$, see Corollary~\ref{r-s = 3 mod 4 (2)}, item (3-2-1). Thus any admissible module $U$ is decomposed on the direct sum of the type~\eqref{eq:rseq3}. 

\begin{lemma}\label{lem:121652}
The Lie algebras 
$$
\n_{r,s}(V_{min;+}^{r,s;-})\stackrel{\Phi_1}\cong\n_{r,s}(V_{min;+}^{r,s;+})\stackrel{\Phi_2}\cong\n_{r,s}(V_{min;-}^{r,s;-})\stackrel{\Phi_3}\cong\n_{r,s}(V_{min;-}^{r,s;+})
$$
for $(r,s)\in\{(1,2),(1,6),(5,2)\}$ are isomorphic under the maps $\Phi_k=A_k\oplus C$, $k=1,2,3$ with $C=\Id_{\mathbb R^{r,s}}$.
\end{lemma}

\begin{proof}
The existence of the maps $\Phi_1$  and $\Phi_3$ follows from Lemma~\ref{uniqueness(0)-2}. We need only to construct $\Phi_2$.

{\sc Case $\n_{1,2}$}
Let $v\in V_{min;+}^{1,2;+}$ be such that $\la v,v\ra_{V_{min;+}^{1,2;+}}=1$ and $u\in V_{min;-}^{1,2;-}$ with
$\la u,u\ra_{V_{min;-}^{1,2;-}}=-1$. It is possible, since both scalar products are neutral on the module. Then the vectors
$$
\begin{array}{lllll}
&x_{1}=v,\ &x_{2}=J_{z_1}v,\ &x_{3}=J_{z_2}v,\ &x_{4}=J_{z_3}v,
\\
&y_{1}=u,\ &y_{2}=-J_{z_1}u,\ &y_{3}=-J_{z_2}u,\ &y_{4}=-J_{z_3}u,
\end{array}
$$
form the orthonormal bases for $V_{min;+}^{1,2;+}$ and $V_{min;-}^{1,2;-}$, respectively.
We define the correspondence:
$
A\colon x_{i}\mapsto y_{i}$  and $C\colon z_{i}\mapsto z_{i}$ and extend it by linearity
It is easy to check that $\Phi_2=A\oplus\Id_{\mathbb R^{1,2}}$
defines an isomorphism between the Lie algebras
$\n_{1,2}(V_{min;+}^{1,2;+})$ and $\n_{1,2}(V_{min;-}^{1,2;-})$. 

{\sc Case $\n_{1,6}$}.  
This case is similar to $\n_{1,2}$ and 
we construct an isomorphism
$\Phi=A\oplus\Id\colon
\n_{1,6}(V_{min;+}^{1,6;+})\to\n_{1,6}(V_{min;-}^{1,6;-})$. We have  
$J_{\Omega^{1,6}}\equiv \Id$ on $V_{min;+}^{1,6;+}$, that implies 
$P_{3}\equiv -\Id$, see Table~\ref{tab:block3}. It also shows that
$E_{1,6}=E_{1,6}(V_{min;+}^{1,6;+})=\{v\in V_{min;+}^{1,6}~|~P_{1}(v)=v,
P_{2}(v)=v\}$.
Then if necessary, we apply 
Proposition~\ref{orthogonal}, with  
the operators $\Lambda_{1}=J_{z_2}J_{z_4}J_{z_6}$ and $\Lambda_{2}=J_{z_2}J_{z_4}J_{z_7}$ and 
obtain the orthonormal basis of $V_{min;+}^{1,6;+}$ 
from a suitable vector $v\in E_{1,6}(V_{min;+}^{1,6;+})$ 
with $\la v,v\ra_{V_{min};+}^{1,6;+}=1$:  
\begin{equation}\label{eq:basis16p}
\begin{array}{llllll}
& x_1=v, \quad &x_i=J_{z_{i-1}}v, &\quad i=1,\ldots,8, &
\\
&x_j=J_{z_2}J_{z_{j-5}}v, &j=9,\ldots,14 &
x_{k}=J_{z_2}J_{z_4}J_{z_{k-9}}v, & k=15, 16.
\end{array}
\end{equation}
Let $u\in E_{1,6}(V_{min;-}^{1,6;-})\subset {V_{min;-}^{1,6;-}}$ with $\la u,u\ra_{V_{min;-}^{1,6;-}}=-1$. Then 
by the same way as for $V_{min;+}^{1,6;+}$ we obtain the
orthonormal basis of $V_{min;-}^{1,6;-}$:

\begin{equation}\label{eq:basis16m}
\begin{array}{llllll}
& y_1=u, \quad &y_i=-J_{z_{i-1}}u, &\quad i=1,\ldots,8, &
\\
&y_j=J_{z_2}J_{z_{j-5}}u, &j=9,\ldots,14 & y_{k}=-J_{z_2}J_{z_4}J_{z_{k-9}}u, & k=15,16.
\end{array}
\end{equation}
Then as previously, the correspondence 
$A\colon x_i\mapsto y_i,\ i=1,\ldots,16$, 
defines the Lie algebra isomorphism 
$\Phi=A\oplus\Id\colon
\n_{1,6}(V_{min;+}^{1,6;+})\to\n_{1,6}(V_{min;-}^{1,6;-})$, 
since the map $A$ satisfies relation~\eqref{isomorphism relation}.

{\sc Case $\n_{5,2}$.} In this case we can use bases~\eqref{eq:basis16p} and~\eqref{eq:basis16m}, since $PI_{5,2}=PI_{1,6}$. The table of commutators will differ by signs for $z_1,\ldots,z_7$.
\end{proof}


\subsection{Proof of Theorem~\ref{main theorem 2}}


Theorem~\ref{main theorem 2} is concerned with the indices $r\equiv
3\,(\text{mod}~4)$ 
and $s\equiv 0\,(\text{mod}~4)$ and is given in Lemmas~\ref{isomorphism V^{+}_{+}and V^{-}_{-}} and~\ref{lem:307034}. The cases with $s=0$ are classical and the result is known, for instance from~ \cite{BTV,CDKR}, however in order
to accomplish the whole classification of pseudo $H$-type Lie algebras
we must take into account that the Lie algebras $\n_{r,0}(U)$ can admit a negative definite admissible scalar product on $U$. Thus, we can obtain the opposite sign of the restriction of the admissible scalar product on the common 1-eigenspace even for classical cases. 

\begin{proposition}\label{isomorphism V^{+}_{+}and V^{-}_{-}} 
Let $r\equiv 3\,(\text{\em mod}~4)$ 
and $s\equiv 0\,(\text{\em mod}~4)$. 
Then an admissible module $U$ is decomposed into the direct sum of type~\eqref{eq:rseq3}. Moreover
\begin{itemize}
\item[$(1)$] {There is a Lie algebra isomorphism $\Phi\colon\n_{r,s}(V^{r,s;+}_{min;+})\to\n_{r,s}(V_{min;-}^{r,s;-})$ of the form
$\Phi=A\oplus \Id$. There is no isomorphism of the form $\Phi=A\oplus -\Id$ between these algebras.
Analogous results can be stated for the Lie algebras isomorphisms $\n_{r,s}(V^{r,s;-}_{min;+})\to\n_{r,s}(V^{r,s;+}_{min;-})$. 
}
\item[$(2)$]{There is a Lie algebra isomorphism
    $\Phi\colon\n_{r,s}(V^{r,s;+}_{min;+})\to\n_{r,s}(V^{r,s;-}_{min;+})$ 
of the form $\Phi=A\oplus C$ with $\det C<0$ and there is no isomorphism of the form $\Phi=A\oplus \Id$.
Analogous results hold for the Lie algebra isomorphisms
$\n_{r,s}(V^{r,s;+}_{min;-})\to\n_{r,s}(V^{r,s;-}_{min;-})$, $\n_{r,s}(V^{r,s;+}_{min;+})\to\n_{r,s}(V^{r,s;+}_{min;-})$, $\n_{r,s}(V^{r,s;-}_{min;+})\to\n_{r,s}(V^{r,s;-}_{min;-})$.
} 
\end{itemize}
\end{proposition}

\begin{proof}
We start from the proof of the first part. We restrict the consideration to the basic cases
$(r,s)\in\{(3,0),(3,4),(7,0)\}$ because of the periodicity Theorem~\ref{th:periodic}. We also can assume that $C=\Id_{\mathbb R^{r,s}}$. In order to construct an isomorphism $\Phi=A\oplus \Id_{\mathbb R^{r,s}}\colon \n_{r,s}(V^{r,s;+}_{min;+})\to\n_{r,s}(V^{r,s;-}_{min;-})$
we choose $v\in E_{r,s}\subset V^{r,s;+}_{min;+}$ with $\la v,v\ra_{V^{r,s;+}_{min;+}}=1$
and $J_{(\Omega^{r,s})}v=v$ and a vector
$u\in E_{r,s}\subset V^{r,s;-}_{min;-}$ with $\la u,u\ra_{V^{r,s;-}_{min;-}}=-1$ and
$\tilde J_{(\Omega^{r,s})}u=-u$. Here $\Omega_{r,s}$ is the volume form of the Clifford algebra $\Cl_{r,s}$ with actions $J\colon \Cl_{r,s}\to\End(V^{r,s;+}_{min;+})$ and $\tilde J\colon \Cl_{r,s}\to\End(V^{r,s;-}_{min;-})$.

Let $(r,s)=(3,0)$. The respective orthonormal bases are the following:
$$
\begin{array}{lllll}
& x_{1}=v,\ x_{2}=J_{z_1}v,\ x_{3}=J_{z_2}v,\ x_{4}=J_{z_3}(v)\quad&\text{for}\quad V^{3,0;+}_{min;+}\quad\text{and}
\\
& y_{1}=u,\ y_{2}=-\tilde{J}_{z_1}u,\ y_{3}=-\tilde{J}_{z_2},\ y_{4}=-\tilde{J}_{z_3}u\quad&\text{for}\quad V^{3,0;-}_{min;-}.
\end{array}
$$

In the case $(r,s)=(7,0)$ the initial vector $v\in E_{7,0}\subset V^{7,0;+}_{min;+}$ satisfies $P_1v=P_2v=P_3v=P_4v=v$, where the involutions are given in Table~\ref{tab:com0r47}. 
Note that $J_{(\Omega^{7,0})}=P_1P_4v=v$. The initial vector $u\in E_{7,0}\subset V^{7,0;-}_{min;-}$ for the basis has to satisfy 
$P_1u=P_2u=P_3u=-P_4u=u$ with $\tilde J_{(\Omega^{7,0})}u=P_1P_4u=-u$. 
The bases are
\begin{equation}\label{eq:basis70}
\begin{array}{lllll}
&x_1=v, \quad &x_j=J_{z_{j-1}}v, \quad j=2,\ldots, 8\quad&\text{for}\quad V^{7,0;+}_{min;+}\quad\text{and}
\\
&y_1=u, &y_j=-\tilde J_{z_{j-1}}u,\quad j=2,\ldots,8, \quad&\text{for}\quad V^{7,0;-}_{min;-}.
\end{array}
\end{equation}

Let $(r,s)=(3,4)$. The basis is given by~\eqref{eq:basis70} and $\Omega^{3,4}=P_1P_3P_4$, where the involutions $P_i$ are presented in Table~\ref{tab:com7s}.

The maps $\Phi=A\oplus\Id_{\mathbb R^{r,s}}$ in all the cases are given 
by correspondence $A\colon x_i\mapsto y_i$. The Lie algebra isomorphisms 
$\Phi=A\oplus \Id_{\mathbb R^{r,s}}\colon \n_{r,s}(V^{r,s;-}_{min;+})\to \n_{r,s}(V^{r,s;+}_{min;-})$ are constructed analogously. 

Assume that the isomorphism $\Phi=A\,\oplus\, C$ with $\det C<0$ exists, where $A\colon V^{r,s;+}_{min;+}\to V^{r,s;-}_{min;-}$. If $(r,s)\in \{(3,0),(7,0)\}$, then $A^{\tau}=-{^t\hspace{-0.09cm}A}$. and  $\Omega^{r,0}=-\tilde \Omega^{r,0}=\Id$ since the minimal admissible modules correspond to the non-equivalent irreducible modules. 
Then 
\[
{^tA}A=-{^tA}\tilde \Omega^{r,0} 
A=A^{\tau}\tilde \Omega^{r,0} A=\prod_{i=1}^{7}J_{C(z_i)}=(\det {^tC})\Omega^{r,0}=\det C<0,
\]
by~\eqref{isomorphism relation}. This is a contradiction, 
since the matrix ${^tA}A$ is positive definite.

Let $(r,s)=(3,4)$. The admissible scalar products restricted to common 1-eigenspace $E_{r,s}$ are sign definite and the  
symmetric bi-linear forms
$\la \tilde x,\tilde{y}\ra_{V^{r,s;-}_{min;-}}$ and 
$\la x,y\ra_{V^{r,s;+}_{min;+}}$ restricted to the common 1-eigenspaces
are related through the equalities
\begin{eqnarray}\label{eq:detC}
\la \tilde x,\tilde y\ra_{V^{r,s;-}_{min;-}}
&=&
\la Ax,Ay\ra_{V^{r,s;-}_{min;-}}=-\la\tilde\Omega^{3,4} Ax,Ay\ra_{V^{r,s;-}_{min;-}}
=-\la A^{\tau} \tilde\Omega^{3,4} Ax,y\ra_{V^{r,s;+}_{min;+}}
\nonumber
\\
&=&
- \la\prod_{i=1}^{7}J_{C(z_i)}x,y\ra_{V^{r,s;+}_{min;+}}
=
-\det C^{\tau}\, \la\Omega^{3,4}x,y\ra_{V^{r,s;+}_{min;+}}
\\
&=&
-\det C^{\tau}\, \la x,y\ra_{V^{r,s;+}_{min;+}}.\nonumber
\end{eqnarray}
The signs of the values of two symmetric bi-linear forms
coincide if $\det C<0$ and 
opposite if $\det C>0$. We conclude that there is no Lie algebra isomorphism $\Phi\colon\n_{r,s}(V^{r,s;+}_{min;+})\to\n_{r,s}(V^{r,s;-}_{min;-})$ of the form $\Phi=A\oplus -\Id_{\mathbb R^{r,s}}$. Remind that the map $A$ maps the common
  $1$-eigenspace from $V^{r,s;+}_{min;+}$ to common $1$-eigenspace
  from $V^{r,s;-}_{min;+}$ by the construction described in 
Section~\ref{section Lie algebra isomorphism} after Proposition~\ref{Lie algebra isomorphism 3}.

The results for the Lie algebras $\n_{r,s}(V^{r,s;-}_{min;+})$ and $\n_{r,s}(V^{r,s;+}_{min;-})$ can be shown similarly.

We prove now the second part of the lemma. The isomorphisms
$$
\n_{r,s}(V^{r,s;+}_{min+})\cong \n_{r,s}(V^{r,s;-}_{min;+}) \quad\text{and}\quad \n_{r,s}(V^{r,s;+}_{min;-})\cong\n_{r,s}(V^{r,s;-}_{min;-})
$$
under the map $\Phi=\Id\oplus-\Id_{\mathbb R^{r,s}}$ are given in 
Lemma~\ref{isomorphism between opposit sign scalar product}. The reader can find the isomorphisms of the form $\Phi=A\oplus\, C$ with $\det C<0$ for the Lie algebras 
$$
\n_{r,s}(V^{r,s;+}_{min+})\cong \n_{r,s}(V^{r,s;+}_{min;-}) \quad\text{and}\quad \n_{r,s}(V^{r,s;-}_{min;+})\cong\n_{r,s}(V^{r,s;-}_{min;-})
$$
in~\cite[Theorem 12]{FM1}. The non existence results are proved in a similar way as for the part 1 of the Lemma. 
\end{proof}

We state separately a corollary of Proposition~\ref{isomorphism V^{+}_{+}and V^{-}_{-}} that is core for the proof of Theorem~\ref{main theorem 2}. 

\begin{corollary}\label{isomorphism V^{+}_{+}and V^{-}_{-}Corollary}
There exists a Lie algebra isomorphism $\n_{r,s}(V^{r,s;+}_{min;+})\to\n_{r,s}(V_{min;-}^{r,s;-})$ of the form
$\Phi=A\oplus \Id$. There does not exist a Lie algebra isomorphism $\n_{r,s}(V^{r,s;+}_{min;+})\to\n_{r,s}(V^{r,s;-}_{min;+})$ of the form $\Phi=A\oplus \Id$.
\end{corollary}

\begin{lemma}\label{lem:307034}
The Lie algebras $\n_{r,s}(U)$, $(r,s)\in\{(3,0),(7,0),(3,4)\}$, are completely determined by the pair of numbers $(p=p^+_{+}+p^-_{-}$, $q=p^-_{+}+p_{-}^+)$
in the decomposition of the admissible module $U$:
$$
U=\big(\stackrel{p^+_+}\oplus V_{min;+}^{r,s;+}\big)\oplus\big(\stackrel{p^{-}_+}\oplus V_{min;+}^{r,s;-}\big)
\oplus \big(\stackrel{p^+_-}\oplus V_{min;-}^{r,s;+}\big)\oplus\big(\stackrel{p^{-}_-}\oplus V_{min;-}^{r,s;-}\big).
$$
\end{lemma}
\begin{proof}
Let $U$ be an arbitrary 
admissible module of $\Cl_{r,s}$ for $(r,s)\in\{(3,0),(7,0),(3,4)\}$.  
First we decompose $U$ into the sum of irreducible modules:
\[
U=\big(\stackrel{p_+}\oplus V_{irr;+}^{r,s}\big)\oplus
\big(\stackrel{p_-}\oplus V_{irr;-}^{r,s}\big)=U_{+} \oplus U_{-}.
\]
Then, $\Omega^{r,s}=\Id_{U_{+}}$ and $\Omega^{r,s}=-\Id_{U_{-}}$.
We decompose the submodules $U_{+}$ and $U_{-}$ into the minimal admissible modules
\[
U_{+}=\big(\stackrel{p_{+}^+}\oplus V^{r,s;+}_{min;+}\big)\oplus \big(\stackrel{p_{+}^-}\oplus V^{r,s;-}_{min;+}\big)\quad
\text{and}\quad
U_{-}=\big(\stackrel{p_{-}^+}\oplus V^{r,s;+}_{min;-}\big)\oplus \big(\stackrel{p_{-}^-}\oplus V^{r,s;-}_{min-}\big),
\]
where $V^{r,s;+}_{min;\pm}$ are the minimal admissible modules for which the restriction of the admissible scalar product on the common 1-eigenspace $E_{r,s}$ is positive definite and  $V^{r,s;-}_{min;\pm}$ are those where the restriction of the scalar product on $E_{r,s}$ is negative definite. It was stated in Proposition~\ref{isomorphism V^{+}_{+}and V^{-}_{-}}, item (1), and Corollary~\ref{isomorphism V^{+}_{+}and V^{-}_{-}Corollary} that
\[
\n_{r,s}(U)\stackrel{\Phi}\cong 
\n_{r,s}(U^+)\quad\text{with}\quad U^+=\big(\stackrel{p=p_{+}^++p_{-}^-}\oplus V_{min;+}^{r,s;+}\big)\oplus\big(\stackrel{q=p_{+}^-+p_{-}^+}\oplus V_{min;-}^{r,s;+}\big),
\]
where $\Phi=A\oplus \Id$. Thus we can consider only the case  when the restrictions on $E_{r,s}$ of the scalar product is positive definite. 
Since the isomorphism between Lie algebras $\n_{r,s}(V_{min;+}^{r,s;+})$ and $\n_{r,s}(V_{min;-}^{r,s;+})$ can not admit the form $A\oplus \Id$ by Corollary~\ref{isomorphism V^{+}_{+}and V^{-}_{-}Corollary}, we conclude that two Lie algebras $\n_{r,s}(U)$ and $\n_{r,s}(\tilde U)$ with 
$$
U=\big(\stackrel{p}\oplus V_{min;+}^{r,s;+}\big)\oplus\big(\stackrel{q}\oplus V_{min;-}^{r,s;+}\big)\quad \text{and}\quad  \tilde U=\big(\stackrel{\tilde p}\oplus V_{min;+}^{r,s;+}\big)\oplus\big(\stackrel{\tilde q}\oplus V_{min;-}^{r,s;+}\big)
$$
are isomorphic if and only if either $(p,q)=(\tilde p,\tilde q)$ or $(p,q)=(\tilde q,\tilde p)$.
\end{proof}

 
\subsection{Proof of Theorem~\ref{main theorem 3}}


In order to prove Theorem~\ref{main theorem 3}, we consider the low values of the indices:
$r=3$, $s=1,2,3,5,6,7$ or $r=7$, $s=1,2,3$ and then we apply the periodicity Theorem~\ref{th:periodic}. An admissible module $U$ has decomposition of the type~\eqref{eq:rsneq3}. Theorem~\ref{th:common_sign} shows that the restriction of the scalar product $\la\cdot\,,\cdot\ra_{V^{r,s}_{min}}$ on the common 1-eigenspace $E_{3,s}$ for $s=1,2,3,4,5,6,7$ and $E_{7,s}$, $s=1,2,3$ is sign definite. We denote by $V^{r,s;+}_{\min}=(V^{r,s;+}_{\min},\la\cdot\,,\cdot\ra_{V^{r,s;+}_{min}})$ the minimal admissible module with positive definite metric on $E_{r,s}$. Analogously, we write $V^{r,s;-}_{min}=(V^{r,s;-}_{min},\la\cdot\,,\cdot\ra_{V^{r,s;-}_{min}})$ for the minimal admissible module with negative definite metric on $E_{r,s}$. We note that all the systems $PI_{r,s}$ include the involution  $P=J_{z_1}J_{z_2}J_{z_3}$ for mentioned values of $r$ and $s$. 

Consider pseudo $H$-type algebras $\n_{3,s}(V^{3,s;+}_{min})$ and $\n_{3,s}(V^{3,s;-}_{min})$ for $s=1,2,3$. The natural inclusion $\mathbb R^{3,0}\subset\mathbb R^{3,s}$ allows to consider the module $V^{3,s;+}_{min}$ as the admissible module $U=V^{3,0;+}_{min;+}\oplus V^{3,0;-}_{min;-}$ of $\Cl_{3,0}$. The module $U$ have to include both eigenspaces of the involution $P=J_{z_1}J_{z_2}J_{z_3}$ and therefore $U=V^{3,s;+}_{min}$ includes both irreducible modules $V^{3,0}_{min;\pm}$. The metric on $U=V^{3,s;+}_{min}$ is neutral and therefore the irreducible modules $V^{3,0}_{min;\pm}$ have to carry the definite metrics of opposite signs. Analogous considerations can be done with $V^{3,s;-}_{min}$. This and the item (1) of Lemma~\ref{isomorphism V^{+}_{+}and V^{-}_{-}} imply the existence of the isomorphisms
\begin{align*}
&A_1\oplus\Id\colon \n_{3,0}(V_{min}^{3,s;+})\to\n_{3,0}(V_{min;+}^{3,0;+}\oplus V_{min;-}^{3,0;-})
\to \n_{3,0}(\stackrel{2}\oplus V_{min;+}^{3,0;+}),\\
&A_2\oplus\Id\colon \n_{3,0}(V_{min}^{3,s;-})\to\n_{3,0}(V_{min;+}^{3,0;-}\oplus V_{min;-}^{3,0;+})
\to \n_{3,0}(\stackrel{2}\oplus V_{min;-}^{3,0;+}).
\end{align*}
for any $s=1,2,3$. More generally there are Lie algebra isomorphisms 
\begin{eqnarray*}
\n_{3,0}(U)&=&\n_{3,0}\big((\stackrel{p^+}\oplus V_{min}^{3,s;+})\bigoplus (\stackrel{p^-}\oplus V_{min}^{3,s;-})\big)\cong
\n_{3,0}\big((\stackrel{2p^+}\oplus V_{min;+}^{3,0;+})\bigoplus (\stackrel{2p^-}\oplus V_{min;-}^{3,s;+})\big)
\\
&=&\n_{3,0}(U).
\end{eqnarray*}
The projection map
$\pi\colon\mathbb{R}^{3,s}\to \mathbb{R}^{3,0}$, where $z_i$, $i=3+1,\ldots 3+s$, are
mapped to zero, defines a Lie algebra surjective homomorphism 
$
\Id\oplus\pi\colon\n_{3,s}(U)\to \n_{3,0}(U)$.
We assume now that there exists a Lie algebra isomorphism 
$$
A\oplus C\colon\n_{3,s}(U)\to\n_{3,s}(\widetilde{U}),\quad\text{where}\quad\widetilde U=\big(\stackrel{\tilde p^+}\oplus V_{min}^{3,s;+}\big)\oplus \big(\stackrel{\tilde p^-}\oplus V_{min}^{3,s;-}\big).
$$
Then it defines a Lie  algebra isomorphism 
$
A\oplus C^{'}\colon\n_{3,0}(U)\to\n_{3,0}(\widetilde{U})$ where $C^{'}$ is the restriction $C\vert_{\mathbb{R}^{3,0}}$ of the map $C$ on $\mathbb R^{3,0}$. We define 
$$
\pi^{'}\colon\mathbb{R}^{3,s}\to\mathbb{R}^{3,0}\quad\text{as}\quad 
\pi'(C(z_{i}))=
\begin{cases}C'(z_{i})\  &\text{for}\ i=1,2,3,
\\
0&\text{for}\ i=3+1,\ldots,3+s.
\end{cases}
$$
Then the diagram 
\begin{equation}\label{CD1}
\begin{CD}
\n_{3,s}(U)@>{\Id\oplus \pi}>>\n_{3,0}(U)\\
@V{A\oplus C}VV @VV{A\oplus C^{'}}V\\
\n_{3,s}(\tilde{U})@>{\Id\oplus \pi^{'}}>>\n_{3,0}(\tilde{U})
\end{CD}
\end{equation}
commuts. We conclude that $\n_{3,s}(U)$ and $\n_{3,s}(\widetilde{U})$ are isomorphic,
if and only if $p^+=\tilde p^+$ and $p^-=\tilde p^-$ or $p^+=\tilde
p^-$ and $p^-=\tilde p^+$ by Theorem~\ref{main theorem 2}, see also
\cite{BTV}.    

In order to prove the reduction of $\n_{3,s}(V^{3,s;\pm}_{min})$ to $\n_{3,0}(V^{3,0;\pm}_{min;\pm})$ for $s=5,6,7$ we  observe that there are Lie algebra isomorphisms 
\begin{align*}
&A_1\oplus\Id\colon \n_{3,0}(V_{min}^{3,s;+})\to\n_{3,0}\big((\stackrel{\alpha(s)}\oplus V_{min;+}^{3,0;+})\oplus (\stackrel{\alpha(s)}\oplus V_{min;-}^{3,0;-})\big)
\to \n_{3,0}(\stackrel{2\alpha(s)}\oplus V_{min;+}^{3,0;+}),\\
&A_2\oplus\Id\colon \n_{3,0}(V_{min}^{3,s;-})\to\n_{3,0}\big((\stackrel{\alpha(s)}\oplus V_{min;+}^{3,0;-})\oplus (\stackrel{\alpha(s)}\oplus V_{min;-}^{3,0;+})\big)
\to \n_{3,0}(\stackrel{2\alpha(s)}\oplus V_{min;-}^{3,0;+}),
\end{align*}
where $\alpha(5)=2$, $\alpha(6)=4$ and $\alpha(7)=8$. The rest of the proof is made analogously. 

Literally the same reduction is made for the Lie algebras $\n_{7,s}(V^{7,s;\pm}_{min})$ to the Lie algebra  $\n_{7,0}(V^{7,0;\pm}_{min;\pm})$ for $s=1,2,3$. The theorem is proved.


\subsection{Final result of the classification}\label{sec:final}


The following statement can be proved analogously to Theorem~\ref{th:periodic}.
\begin{theorem}\label{th:periodic1}
If the Lie algebras $\n_{r,s}(U^{r,s})$ and  $\n_{s,r}(\tilde U^{s,r})$ are isomorphic, then the Lie algebras $\n_{r+\mu,s+\nu}(U^{r+\mu,s+\nu})$ and $\n_{s+\nu,r+\mu}(\tilde U^{s+\nu,r+\mu})$ are also isomorphic for $(\mu,\nu)\in \{(8,0),(0,8),(4,4)\}$.
\end{theorem}

We showed in Theorem~\ref{main theorem 1} that the Lie algebras
$\n_{r,s}(U)$ for $r\equiv 0,1,2\,(\text{mod}~4)$ are defined by the
dimension of the admissible module $U$. 
If $r\equiv 3\,(\text{mod}~4)$ the Lie algebra $\n_{r,s}(U)$ depends on the decompositions 
\begin{equation}\label{eq:isotyp}
U=\big(\stackrel{p}\oplus V^{r,s}_{min;+}\big)\bigoplus \big(\stackrel{q}\oplus V^{r,s}_{min;-}\big)\quad\text{or}\quad 
U=\big(\stackrel{p}\oplus V^{r,s;+}_{min}\big)\bigoplus \big(\stackrel{q}\oplus V^{r,s;-}_{min}\big),
\end{equation}
where the numbers $p,q$ are defined in Theorems~\ref{main theorem 2} and~\ref{main theorem 3}. We call admissible modules with decompositions~\eqref{eq:isotyp} {\it isotypic} if one of the numbers $p$ or $q$ vanishes. Otherwise the admissible module is called {\it nonisotypic}.

\begin{theorem}\label{th:finalclass}
Let $r\equiv 0,1,2\,(\text{\em mod}~4)$ and 
$s\equiv 0,1,2\,(\text{\em mod}~4)$. 
Then $\n_{r,s}(U)\cong\n_{s,r}(\tilde U)$ if $\dim(U)=\dim(\tilde U)$.

Let $r\equiv 3\,(\text{\em mod}~8)$, $s\equiv 0,4,5,6\,(\text{\em mod}~8)$
or 
$r\equiv 7\,(\text{\em mod}~8)$, $s\equiv 0,1,2\,(\text{\em mod}~8)$. 
Then $\n_{r,s}(U)\cong\n_{s,r}(\tilde U)$ if $\dim(U)=\dim(\tilde U)$ 
and $U$ is an isotypic admissible module.

Let $r\equiv 3\,(\text{\em mod}~8)$ and $s\equiv 1,2,7\,(\text{\em mod}~8)$. 
Then $\n_{3,s}(U)$ is never isomorphic to $\n_{s,3}(\tilde U)$. 
\end{theorem}

We summarise the results of Theorem~\ref{th:finalclass} in Table~\ref{t:step4}. We distinguish the columns for $r=3$ and $r=7$ for isotypic and nonisotypic modules. We write the symbol $\cong$ in the place $(r,s)$ if $\n_{r,s}(U)$ is isomorphic to $\n_{s,r}(\tilde U)$ and the isomorphism only depends on the dimension of the admissible module. For example, $\n_{3,0}(U)\cong\n_{0,3}(\tilde U)$ if $\dim(U)=\dim(\tilde U)$ and $U$ is isotypic and $\n_{3,0}(U)\not\cong\n_{0,3}(\tilde U)$ if $U$ is non-isotypic. We have $\n_{3,1}(U)\not\cong\n_{1,3}(\tilde U)$ for any admissible modules $U$ and $\tilde U$ even if $\dim(U)=\dim(\tilde U)$ and the module $U$ of $\Cl_{3,1}$ is isotypic: $U=\stackrel{p}\oplus V^{3,1;+}_{\min}$ or $U=\stackrel{q}\oplus V^{3,1;-}_{\min}$.
\begin{table}[htbp]
\center\caption{Final result of the classification}
\vskip0.3cm
\tiny
\begin{tabular}{|c||c|c|c|c|c|c|c|c|c|c|c|}
\hline
$8$&$\cong$&$\cong$&$\cong$&$\cong$&$\cong$&$\cong$&$\cong$&$\cong$&$\cong$&$\cong$&$\circlearrowright$
\\
\hline
7&$\cong$&$\cong$&$\cong$&$\not\cong$&$\not\cong$&$\cong$&$\not\cong$&$\not\cong$&$\not\circlearrowright$&$\not\circlearrowright$&$\cong$
\\
\hline
6&$\cong$&$\cong$&$\cong$&$\cong$&$\not\cong$&$\cong$&$\cong$&$\circlearrowright$&$\not\cong$&$\not\cong$&$\cong$
\\
\hline
5&$\cong$&$\cong$&$\cong$&$\cong$&$\not\cong$&$\cong$&$\circlearrowright$&$\cong$&$\not\cong$&$\not\cong$&$\cong$
\\
\hline
4&$\cong$&$\cong$&$\cong$&$\cong$&$\not\cong$&$\circlearrowright$&$\cong$&$\cong$&$\cong$&$\not\cong$&$\cong$
\\
\hline
3&$\cong$&$\not\cong$&$\not\cong$&$\not\circlearrowright$&$\not\circlearrowright$&$\cong$&$\cong$&$\cong$&$\not\cong$&$\not\cong$&$\cong$
\\
\hline
2&$\cong$&$\cong$&$\circlearrowright$&$\not\cong$&$\not\cong$&$\cong$&$\cong$&$\cong$&$\cong$&$\not\cong$&$\cong$
\\
\hline
1&$\cong$&$\circlearrowright$&$\cong$&$\not\cong$&$\not\cong$&$\cong$&$\cong$&$\cong$&$\cong$&$\not\cong$&$\cong$
\\
\hline
0&&$\cong$&$\cong$&$\cong$&$\not\cong$&$\cong$&$\cong$&$\cong$&$\cong$&$\not\cong$&$\cong$
\\
\hline\hline
$s/r$&0&1&2&3\ \text{isotyp}&3\ \text{nonisotyp}&4&5&6&7\ \text{isotyp}&7\ \text{nonisotyp} &8\\
\hline
\end{tabular}\label{t:step4}
\end{table}


\section{Appendix}\label{app}


We give the collections $PI_{r,s}$ 
and $CO_{r,s}$ for basic cases~\eqref{basic cases} grouped in four tables. The
dimensions of $E_{r,s}$ and signature of the scalar product restricted to $E_{r,s}$
are listed.
First we mention trivial cases.
\begin{align*}
&PI_{1,0}=PI_{0,1}=PI_{2,0}=PI_{1,1}=PI_{0,2}=PI_{2,1}=PI_{0,3}=\emptyset,\\
&PI_{3,0}=PI_{1,2}=\{P=J_{z_1}J_{z_2}J_{z_3}\},\quad CO_{3,0}=CO_{1,2}=\emptyset
\end{align*}
For the cases $(r,s)$ of $r-s\equiv 3 \,(\text{mod}~4)$ and $s$ even,
there are no a complementary operator which commutes with all the
involutions in $PI_{r,s}$ except the last involution which is of the form $\mathcal P_3$ or $\mathcal P_4$ and anti-commutes with the last involution.
In these cases the operator $J_{\Omega^{r,s}}$ is a product of
involutions in $PI_{r,s}$ and it commutes with all the
operators $J_{z}$. This is the reason for the number of complementary operators to be $p_{r,s}-1$. The last operator in $PI_{r,s}$ of the form $\mathcal P_3$ or $\mathcal P_4$ distinguishes the two different minimal admissible modules.

The signature on the admissible scalar product restricted on 
the space $E_{r,s}$ is sign definite in Table~\ref{tab:com7s} and 
is neutral for signatures $(r,s)$ in Table~\ref{tab:block3}.
The latter can be seen by finding an additional negative operator other than
 operators in $CO_{r,s}$ which commutes with
all the involutions in $PI_{r,s}$.

\begin{table}[htbp]
\caption{Systems $PI_{r,0}$ and $CO_{r,0}$, $r=4,\ldots,7$}
\centering
\tiny
\begin{tabular}{|c|c| c | c | c | c | c || c | c | c | c | c | c | c |} 
\hline
 $PI_{r,0}$$\setminus CO_{r,0}$  & $\text{isom}$ &$\text{isom}$
  &$\text{isom}$ & $\dim(E_{r,0})\ {\text{and signature}}$
 \\
\hline
 $PI_{4,0}$$\setminus CO_{4,0}$  & $J_{z_1}$ &$$ &$$ & 
 \\
\hline
$P_1=J_{z_1}J_{z_2}J_{z_3}J_{z_4}$  & $-1$ & $$ & $$ &$\dim(E_{4,0})=4,\ \pm$
\\
\hline\hline
 $PI_{5,0}$$\setminus CO_{5,0}$  & $J_{z_2}$ &$J_{z_1}$ &$$& 
 \\
\hline
$P_1=J_{z_2}J_{z_3}J_{z_{4}}J_{z_5}$  & $-1$ & $1$ & $$ &
\\
\hline
$P_2=J_{z_1}J_{z_2}J_{z_3}$  & $$ & $-1$ & $$ &$\dim(E_{5,0})=2,\ \pm$
\\ 
\hline\hline
 $PI_{6,0}$$\setminus CO_{6,0}$  & $J_{z_1}$ &$J_{z_5}$ &$J_{z_2}J_{z_4}$ &
 \\
\hline
$P_1=J_{z_1}J_{z_2}J_{z_3}J_{z_{4}}$  & $-1$ & $1$ & $1$ &
\\
\hline
$P_2=J_{z_1}J_{z_2}J_{z_5}J_{z_6}$  & $$ & $-1$ & $1$  &
\\
\hline   
$P_3=J_{z_1}J_{z_3}J_{z_5}$ & $$ & $$ & $-1$ &$\dim(E_{6,0})=1, \pm$
\\
\hline\hline
$PI_{7,0}$$\setminus CO_{7,0}$  & $J_{z_1}$ &$J_{z_1}J_{z_3}$ &$J_{z_1}J_{z_2}$& 
 \\
\hline
$P_1=J_{z_1}J_{z_2}J_{z_3}J_{z_4}$  & $-1$ & $1$ & $1$ &
\\
\hline
$P_2=J_{z_1}J_{z_2}J_{z_5}J_{z_6}$  & $$ & $-1$ & $1$  &
\\
\hline
$P_3=J_{z_1}J_{z_3}J_{z_5}J_{z_7}$ & $$ & $$ & $-1$ & $\dim(E_{7,0})=1,\ \pm$
\\
\hline
$P_4=J_{z_5}J_{z_6}J_{z_7}$ & $$ & $$ & $$ &
\\
\hline
\end{tabular}\label{tab:com0r47}
\end{table}

\begin{table}[htbp]
\caption{Systems $PI_{r,4}$ and $CO_{r,4}$, $r=0,1,2$}
\centering
\tiny
\begin{tabular}{| c | c | c | c | c | c | c | c | c | c | c | c | c | c |} 
\hline
 $PI_{r,4}$$\setminus CO_{r,4}$  & $\text{isom}$& $\text{anti-isom}$
  &$\text{isom}$ &$\dim(E_{r,4}) \ \text{and signature}$& $\text{basis for}\ E_{r,4}$
 \\
\hline
 $PI_{0,4}=PI_{4,0}$$\setminus CO_{0,4}$  & $$ &$J_{z_2}$ &$$ & $$& $$
 \\
\hline
$P_1=J_{z_1}J_{z_2}J_{z_3}J_{z_4}$  & $$ & $-1$ & $$ &$\dim(E_{0,4})=4,\ \pm$ & $\begin{array}{llll} &v, & J_{z_1}J_{z_2}v,\\ &J_{z_1}J_{z_3}v,& J_{z_1}J_{z_4}v\end{array}$
\\
\hline\hline
 $PI_{1,4}=PI_{5,0}$$\setminus CO_{1,4}$  & $$ & $J_{z_{2}}$ &$J_{z_1}$& $$&
 \\
\hline
$P_1=J_{z_2}J_{z_3}J_{z_4}J_{z_{5}}$  & $$ & $-1$ & $1$ &  $$ &
\\
\hline
$P_2=J_{z_1}J_{z_2}J_{z_3}$  & $$ & $$ & $-1$ &$\dim(E_{1,4})=2,\ \pm$& $v,\quad J_{z_1}v$
\\ 
\hline\hline
 $PI_{2,4}=PI_{6,0}$$\setminus CO_{2,4}$  & $J_{z_1}$ &$J_{z_5}$ &$J_{z_1}J_{z_{2}}$&$$&
 \\
\hline
$P_1=J_{z_1}J_{z_2}J_{z_3}J_{z_4}$  & $-1$ & $1$&  $1$& $$ &
\\
\hline
$P_2=J_{z_1}J_{z_2}J_{z_5}J_{z_6}$  & $$ & $-1$& $1$&  $$ &
\\
\hline
$P_3=J_{z_1}J_{z_3}J_{z_5}$ & $$ & $$& $-1$& $\dim(E_{2,4})=1,\ \pm$ & $v$ 
\\
\hline
\end{tabular}\label{tab:comr4}
\end{table} 

\begin{table}[htbp]
\caption{Systems $PI_{3,s}$ and $CO_{3,s}$, $s=1,\ldots,7$ and $PI_{7,s}$, $CO_{7,s}$, $s=1,2,3$}
\centering
 \tiny
\begin{tabular}{| c | c | c | c | c | c | c | c | c | c | c | c | c | c | c |} 
\hline
 $PI_{r,s}$$\setminus CO_{r,s}$  & $\text{isom}$ & $\text{isom}$ &$\text{isom}$
  &$\text{anti-isom}$ & 
$\dim(E_{3,s})$& $\text{basis for}\ E_{3,s}$
 \\
\hline
 $PI_{3,1}$$\setminus CO_{3,1}$ &  &$$ & $$& $J_{z_4}$ &$$ &$$
 \\
\hline
$P_1=J_{z_1}J_{z_2}J_{z_3}$ &  & $$ & $$ & $-1$ &$\dim(E_{3,1})=4$ & $v,J_{z_1}v,J_{z_2}v,J_{z_3}v$ 
\\
\hline\hline
 $PI_{3,2}$$\setminus CO_{3,2}$ &  & $$ &$J_{z_1}$ &$J_{z_2}J_{z_4}$& $$&
 \\
\hline
$P_1=J_{z_1}J_{z_2}J_{z_4}J_{z_5}$ &  & $$ & $-1$ & $1$ & $$&
\\
\hline
$P_2=J_{z_1}J_{z_2}J_{z_3}$ &  & $$ & $$ & $-1$ &$\dim(E_{3,2})=2$& $v,J_{z_3}v$
\\ 
\hline\hline
 $PI_{3,3}$$\setminus CO_{3,3}$ &  & $J_{z_1}$ &$J_{z_3}$ &$J_{3}J_{z_6}$&$$&
 \\
\hline
$P_1=J_{z_1}J_{z_2}J_{z_4}J_{z_5}$ &  & $-1$ & $1$ & $1$& $$ &
\\
\hline
$P_2=J_{z_1}J_{z_3}J_{z_4}J_{z_{6}}$ &  & $$ & $-1$ & $1$&  $$ &
\\
\hline
$P_3=J_{z_1}J_{z_2}J_{z_3}$ & & $$ & $$ & $-1$& $\dim(E_{3,3})=1$ & $v$ 
\\
\hline\hline
$PI_{3,4}$$\setminus CO_{3,4}$ &  & $J_{z_1}$ &$J_{z_3}$ &$J_{z_6}$&$$&
 \\
\hline
$P_1=J_{z_1}J_{z_2}J_{z_4}J_{z_5}$ &  & $-1$ & $1$ & $1$ & $$ &
\\
\hline
$P_2=J_{z_1}J_{z_3}J_{z_5}J_{z_7}$ &  & $$ & $-1$ & $1$  & $$ &
\\
\hline
$P_3=J_{z_1}J_{z_2}J_{z_6}J_{z_{7}}$ & & $$ & $$ & $-1$& $\dim(E_{3,4})=1$&$v$ 
\\
\hline
$P_4=J_{z_1}J_{z_2}J_{z_3}$ & & $$ & $$ & $$ &$$& 
\\
\hline\hline
$PI_{3,5}$$\setminus CO_{3,5}$  & $J_{z_1}$ &$J_{z_6}J_{z_8}$ &$J_{z_3}$& $J_{z_8}$ &$$& $$
 \\
\hline
$P_1=J_{z_1}J_{z_2}J_{z_4}J_{z_5}$  & $-1$ & $1$ & $1$ & $1$& $$ &
\\
\hline
$P_2=J_{z_1}J_{z_2}J_{z_6}J_{z_{7}}$  & $$ & $-1$ & $1$& $1$  &$$ &
\\
\hline
$P_3=J_{z_1}J_{z_3}J_{z_5}J_{z_7}$ & $$ & $$& $-1$ & $1$ & $$&$$ 
\\
\hline
$P_4=J_{z_1}J_{z_2}J_{z_3}$ & $$ & $$ & $$ & $-1$&$\dim(E_{3,5})=1$ & $v$
\\
\hline\hline
$PI_{3,6}=PI_{3,5}$  & $CO_{3,6}=CO_{3,5}$ &$$ &$$& $$ &$\dim(E_{3,6})=2$& $v,\quad J_{z_8}J_{z_9}v$
 \\
\hline\hline
$PI_{3,7}=PI_{3,5}$  & $CO_{3,7}=CO_{3,5}$ &$$ &$$& $$ &$\dim(E_{3,7})=4$& $\begin{array}{lll}
&v, &J_{z_8}J_{z_9}v,
\\
&J_{z_8}J_{z_{10}}v, & J_{z_9}J_{z_{10}}v
\end{array}$
 \\
\hline\hline
$PI_{7,1}$$\setminus CO_{7,1}$  & $J_{z_1}$ &$J_{z_5}$ &$J_{z_7}$& $J_{z_8}$ &$$& $$
 \\
\hline
$P_1=J_{z_1}J_{z_2}J_{z_3}J_{z_4}$  & $-1$ & $1$ & $1$ & $1$& $$ &
\\
\hline
$P_2=J_{z_1}J_{z_2}J_{z_{5}}J_{z_6}$  & $$ & $-1$ & $1$& $1$  &$$ &
\\
\hline
$P_3=J_{z_1}J_{z_3}J_{z_5}J_{z_7}$ & $$ & $$& $-1$ & $1$ & $$&$$ 
\\
\hline
$P_4=J_{z_1}J_{z_2}J_{z_3}$ & $$ & $$ & $$ & $-1$&$\dim(E_{7,1})=1$ & $v$
\\
\hline\hline
$PI_{7,2}=PI_{71}$  & $CO_{7,2}=CO_{7,1}$ &$$ &$$& $$ &$\dim(E_{7,2})=2$& $v,\quad J_{z_8}J_{z_9}v$
\\
\hline\hline
$PI_{7,3}=PI_{71}$& $CO_{7,3}=CO_{7,1}$ &$$ &$$& $$ &$\dim(E_{7,3})=4$& $\begin{array}{lll}
&v, &J_{z_8}J_{z_9}v,
\\
&J_{z_8}J_{z_{10}}v, & J_{z_9}J_{z_{10}}v
\end{array}$
\\
\hline
\end{tabular}\label{tab:com7s}
\end{table}
\begin{table}[htbp]
\caption{Systems $PI_{r,s}$ and $CO_{r,s}$ for Proposition~\ref{th:common_sign}}
\centering
\tiny
\begin{tabular}{| c | c | c | c | c | c | c | c | c | c | c |}
\hline
 $PI_{r,s}$$\setminus CO_{r,s}$  & $\text{isom}$& $\text{isom}$& $\text{isom}$ &$\text{anti-isom}$  &
                                                                   $\dim(E_{r,s})$
 \\ 
\hline
 $PI_{1,3}$$\setminus CO_{1,3}$ &  & $J_{z_2}J_{z_4}$ &$$  & $$&
 \\
\hline
$P_1=J_{z_1}J_{z_2}J_{z_3}$ &  & $-1$ & $$ &&$\dim(E_{1,3})=4$
\\
\hline\hline
 $PI_{2,2}$$\setminus CO_{2,2}$ &  & $J_{z_1}$& $$  &&
 \\
\hline
$P_1=J_{z_1}J_{z_2}J_{z_3}J_{z_4}$&  & $-1$ & $$ & &$\dim(E_{2,2})=4$
\\ 
\hline\hline
 $PI_{2,3}$$\setminus CO_{2,3}$ & & $J_{z_1}$ & $J_{z_1}J_{z_2}$ &&
 \\
\hline
$P_1=J_{z_1}J_{z_2}J_{z_3}J_{z_4}$&  & $-1$ & $1$ &&
\\
\hline
$P_2=J_{z_2}J_{z_3}J_{z_5}$ & & $$ & $-1$ &&$\dim(E_{2,3})=2$
\\
\hline\hline
 $PI_{0,5}$$\setminus CO_{0,5}$  & $J_{z_1}J_{z_5}$ &$$ &$$ &$$ & 
 \\
\hline
$P_1=J_{z_1}J_{z_2}J_{z_3}J_{z_4}$  & $-1$&$$ & $$ & $$ &$\dim(E_{0,5})=8$
\\
\hline\hline
 $PI_{0,6}$$\setminus CO_{0,6}$  & $J_{z_1}J_{z_5}$ &$J_{z_2}J_{z_3}$ &$$&$$& 
 \\
\hline
$P_1=J_{z_1}J_{z_2}J_{z_3}J_{z_4}$  & $-1$ & $1$&$$ & $$ &
\\
\hline
$P_2=J_{z_1}J_{z_2}J_{z_5}J_{z_6}$  & $$ & $-1$&$$ & $$ &$\dim(E_{0,6})=4$
\\ 
\hline\hline
 $PI_{0,7}$$\setminus CO_{0,7}$  & $J_{z_1}J_{z_5}$ &$J_{z_2}J_{z_3}$ &$J_{z_5}J_{z_6}$&$$ &
 \\
\hline
$P_1=J_{z_1}J_{z_2}J_{z_3}J_{z_4}$  & $-1$ & $1$ & $1$&$$ &
\\
\hline
$P_2=J_{z_1}J_{z_2}J_{z_5}J_{z_6}$  & $$ & $-1$ & $1$&$$  &
\\
\hline
$P_3=J_{z_1}J_{z_3}J_{z_5}J_{z_7}$ & $$ & $$ & $-1$&$$ &$\dim(E_{0,7})=2$
\\
\hline\hline
$PI_{1,5}$$\setminus CO_{1,5}$  & $J_{z_5}J_{z_6}$ &$J_{z_3}J_{z_4}$ &$$&$$& 
 \\
\hline
$P_1=J_{z_2}J_{z_3}J_{z_4}J_{z_5}$  & $-1$ & $1$ & $$&$$ &
\\
\hline
$P_2=J_{z_1}J_{z_2}J_{z_3}$ & $$ & $-1$ & $$&$$ & $\dim(E_{1,5})=4$
\\
\hline\hline
$PI_{1,6}$$\setminus CO_{1,6}$  & $J_{z_5}J_{z_6}$ &$J_{z_3}J_{z_4}$ &$$&$J_{z_2}J_{z_4}J_{z_6}$& 
 \\
\hline
$P_1=J_{z_2}J_{z_3}J_{z_4}J_{z_5}$  & $-1$ & $1$ & $$&$1$ &
\\
\hline
$P_2=J_{z_2}J_{z_3}J_{z_6}J_{z_7}$ & $$ & $-1$&$$ & $1$ & $$
\\
\hline
$P_3=J_{z_1}J_{z_2}J_{z_3}$ & $$ & $$ & $$&$1$ & $\dim(E_{1,6})=4$
\\
\hline\hline
$PI_{1,7}\setminus CO_{1,7}$   &$J_{z_5}J_{z_6}$ &$J_{z_3}J_{z_4}$ &$J_{z_2}J_{z_4}J_{z_6}J_{z_8}$&$$&
\\
\hline
$P_1=J_{z_2}J_{z_3}J_{z_4}J_{z_5}$  & $-1$ & $1$ & $1$&$$ &
\\
\hline
$P_2=J_{z_2}J_{z_3}J_{z_6}J_{z_7}$ & $$ & $-1$ & $1$&$$ & $$
\\
\hline
$P_3=J_{z_1}J_{z_2}J_{z_3}$ & $$ & $$ & $-1$&$$ & $\dim(E_{1,7})=4$
\\
\hline\hline
$PI_{2,5}$$\setminus CO_{2,5}$  & $J_{z_1}$ &$J_{z_1}J_{z_3}J_{z_7}$ &$J_{z_5}J_{z_6}$&$$& 
 \\
$P_1=J_{z_1}J_{z_2}J_{z_3}J_{z_4}$  & $-1$ & $1$ & $1$&$$ &
\\
\hline
$P_2=J_{z_1}J_{z_2}J_{z_5}J_{z_6}$ & $$ & $-1$ & $1$&$$ & $$
\\
\hline
$P_3=J_{z_1}J_{z_3}J_{z_5}$ & $$ & $$ & $-1$&$$ & $\dim(E_{2,5})=2$
\\
\hline\hline
$PI_{2,6}=PI_{2,5}$  & $CO_{2,6}=CO_{2,5}$ &$$ &$$&$$ &$\dim(E_{2,6})=4$ 
 \\
\hline
$PI_{2,7}=PI_{2,5}$ & $CO_{2,7}=CO_{2,5}$ &$$ &$$&$$& $\dim(E_{2,7})=8$
\\
\hline
 $PI_{4,1}=PI_{0,5}$ & $CO_{4,1}=\{J_{z_1}\}$ &&&&$\dim(E_{4,1})=8$  
\\
\hline
 $PI_{4,2}=PI_{0,6}$  & $CO_{4,2}=\{J_{z_1}, J_{z_2}J_{z_3}\}$ &&&&$\dim(E_{4,2})=4$  
\\ 
\hline
 $PI_{4,3}=PI_{0,7}$  & $CO_{4,3}=\{J_{z_1}, J_{z_2}J_{z_3}, J_{z_5}J_{z_6}\}$ &&&&$\dim(E_{4,3})=2$ 
 \\
\hline
$PI_{5,1}=PI_{1,5}$  & $CO_{5,1}=\{J_{z_5}, J_{z_3}J_{z_4}\}$ &&&&$\dim(E_{5,1})=4$ 
 \\
\hline
$PI_{5,2}=PI_{1,6}$  & $CO_{5,2}=\{J_{z_5}, J_{z_3}J_{z_4},J_{z_2}J_{z_4}J_{z_6}\}$ &&&&$\dim(E_{5,2})=4$  
\\
\hline
$PI_{5,3}=PI_{1,7}$  & $ CO_{5,3}=\{J_{z_5}, J_{z_3}J_{z_4}, J_{z_2}J_{z_4}J_{z_6}J_{z_8}\}$ & &&&$\dim(E_{5,3})=4$
 \\
\hline
$PI_{6,1}=PI_{2,5}$ & $CO_{6,1}=\{J_{z_1}, J_{z_1}J_{z_3}, J_{z_5}J_{z_6}\}$ &&&&$\dim(E_{6,1})=2$  
 \\
\hline
$PI_{6,2}=PI_{2,5}$  & $CO_{6,2}=CO_{6,1}$ &&&&$\dim(E_{6,2})=4$  
 \\
\hline
$PI_{6,3}=PI_{2,5}$ & $CO_{6,3}=CO_{6,1}$ &&&&$\dim(E_{6,3})=8$
\\
\hline
\end{tabular}\label{tab:block3}
\end{table}


\end{document}